\documentclass[11pt, reqno]{amsart}
\usepackage[utf8]{inputenc}
\usepackage{epsf}  
\usepackage{amsmath, amsthm, amssymb, amsfonts, verbatim, color}
\usepackage{mathrsfs}
\usepackage[pdftex]{graphicx}
\usepackage[bookmarks]{hyperref}

\renewcommand{\baselinestretch}{1}
\reversemarginpar

\def\be{\begin{equation}}
\def\ee{\end{equation}}

\def\H{{\mathbb{H}}}

\def\eps{\varepsilon}

\def\R{{\mathbb R}}

\def\bs{\begin{split}}

\theoremstyle{plain}

\def\lm{\begin{lem}}
\def\ml{\end{lem}}
\newtheorem{thm}{Theorem}[section]
\newtheorem{cor}[thm]{Corollary}
\newtheorem{lem}[thm]{Lemma}
\newtheorem{prop}[thm]{Proposition}

\theoremstyle{definition}

\theoremstyle{remark}
\newtheorem{remark}{Remark}[section]

\numberwithin{equation}{section}

\def\be{\begin{equation}}
\def\ee{\end{equation}}

\def\R{\mathbb R}

\def\pru{\begin{proof}}
\def\urp{\end{proof}}

\theoremstyle{definition}
\newtheorem*{definition}{Definition}

\newtheorem*{theorem*}{Theorem}
\newtheorem*{corollary*}{Corollary}
\newtheorem*{acknowledgment}{Acknowledgments}

\begin{document}


\title[Mild Solutions to Navier-Stokes on the Hyperbolic Space]{Well-posedness and global in time behavior for $L^p$-mild solutions to the Navier-Stokes Equation on the hyperbolic space}

\author[B. Balentine]{Braden Balentine}
\address{Department of Mathematics\\
University of Colorado Boulder\\ Campus Box 395, Boulder, CO, 80309, USA}
%

\email{braden.balentine@colorado.edu}

\begin{abstract}
We study mild solutions to the Navier-Stokes equation on the $n$-dimensional hyperbolic space $\mathbb{H}^n$, $n \geq 2$. We use dispersive and smoothing estimates proved by Pierfelice on a class of complete Riemannian manifolds to extend the Fujita-Kato theory of mild solutions from $\mathbb{R}^n$ to $\mathbb{H}^n$. This includes well-posedness results for $L^n$ initial data and $L^n \cap L^p$ initial data for $1 < p < n$, global in time results for small initial data, and time decay results for the $L^n$ and $L^p$ norms of both $u$ and $\nabla u$. Due to the additional exponential time decay offered on $\mathbb{H}^n$, we are able to simplify the proofs of the $L^n$ and $L^p$ norm decay results as compared to the Euclidean setting. Additionally, we are able to show that mild solutions on $\mathbb{H}^n$ belong to a wider range of space-time $L^rL^q$ spaces than is known for Euclidean space, and that the $L^n$ norm of a global solution decays to zero as $t$ goes to infinity on $\mathbb{H}^n$, which was a question left open by Kato for $\mathbb{R}^n$, $n\geq 3$. As a necessary part of our work, we extend to $\mathbb{H}^n$ known facts in Euclidean space concerning the strong continuity and contractivity of the semigroup generated by the Laplacian. Also, we establish necessary boundedness and commutation properties for a certain projection operator in the setting of $\mathbb{H}^n$ using spectral theory. This work, together with Pierfelice's, contributes to providing a full Fujita-Kato theory on $\mathbb{H}^n$. 

\end{abstract}

\maketitle

\tableofcontents

 
\section{Introduction}\label{intro} 
In this work, we are concerned with the Navier-Stokes equation on the $n$-dimensional hyperbolic space $\mathbb{H}^n$ of constant sectional curvature $-1$, where $n \geq 2$. Specifically, we study the theory of mild solutions to Navier-Stokes on $\mathbb{H}^n$. To begin, let us recall the definition of mild solutions to Navier-Stokes in the context of Euclidean space $\R^n$. 

On $\R^n$, the Navier-Stokes equation is given by
\begin{equation} \tag{$\text{N-S}_{\R^n}$} \label{NSRn}
\begin{split}
\partial_t u - \Delta_{\R^n} u + u \cdot \nabla u + \nabla p &= 0 \\
\text{div } u &= 0, \quad u(0,t) = a, 
\end{split}
\end{equation}
where $u = (u_1, u_2, \ldots, u_n)$ is the velocity of the fluid, $p$ a scalar field denoting the pressure, the condition $\text{div } u = 0$ means the fluid is incompressible, and the term  $u \cdot \nabla u$ is given in coordinates by 
\[
(u \cdot \nabla u)^j = u^i\partial_iu^j,
\]
where we sum over the repeated index $i$.

The mild solution approach was introduced by Fujita and Kato in \cite{kato_fujita} and Sobolevski\u{\i} in \cite{sobolevskii} and involves using the Leray projector $\mathbb{P}_{\R^n}$ onto divergence free vector fields to rewrite equation \eqref{NSRn} as
 \begin{equation}  \tag{$\text{N-S}_{\R^n}'$} \label{NSRn'}
\begin{split}
\partial_t u - \Delta_{\R^n} u +\mathbb{P}_{\R^n}(u \cdot \nabla u)&= 0 \\
\text{div } u &= 0, \quad u(0,t) = a, 
\end{split}
\end{equation}
which can be considered as a non-linear perturbation of the heat equation on $\R^n$ (see, for example, \cite{kato_fujita} and \cite{KatoMild}). Thus by using the heat semigroup $e^{t\Delta_{\R^n}}$ and Duhamel's formula, the equation \eqref{NSRn'} can be converted into an integral equation
\begin{equation}\tag{$\text{N-S}_{\R^n,\text{int}}$} \label{NSRn-int}
u(t) = e^{t\Delta_{\R^n}}a - \int_0^t e^{(t-s)\Delta_{\R^n}}\mathbb{P}_{\R^n}(u(s) \cdot \nabla u(s))\,ds, \quad \text{div } a = 0,
\end{equation}
which is then solved by means of fixed point methods in suitably chosen Banach spaces (or equivalently, by Picard iteration). Solutions to this integral equation are referred to as mild solutions.

We remark here that mild solutions were originally called  ``strong solutions" by Fujita and Kato in \cite{kato_fujita} and as we will see from Kato's results in \cite{KatoMild} , this is because local existence, uniqueness, and smoothness results hold for mild solutions in all dimensions $n\geq 2$ with initial data coming from a wide range of $L^p$ spaces. Additionally, in the case of suitably small initial data, mild solutions are known to exist globally and decay estimates of certain  $L^p$ norms have been established.  We review more of the literature concerning mild solutions below.

The situation for mild solutions stands in contrast to the situation for Leray-Hopf weak solutions of the form
\begin{equation}
u \in L^{\infty}([0,T], L^2(\R^n)) \cap L^{2}([0,T], \dot{H}^1(\R^n)),
\end{equation}
which are weak solutions of Navier-Stokes satisfying the global energy inequality
\begin{equation} \label{enineq}
\|u(t)\|^2_{L^2(\R^n)} + 2 \int_0^t \|\nabla u(s) \|^2_{L^2(\R^n)} \,ds \leq \|a\|^2_{L^2(\R^n)}, \quad 0 \leq t \leq T.
\end{equation}
In the work of Leray \cite{leray1934} and Hopf \cite{Hopf}, global existence of weak solutions was shown for dimensions $n=2,3$ and in the case of $\R^2$, smoothness and uniqueness of these solutions is well-known. However, unlike the situation for mild solutions and small initial data, questions concerning global existence, uniqueness, smoothness, and time decay for $L^2$ norms remain open for Leray-Hopf weak solutions in the case of $\R^3$.

\subsection{The Navier-Stokes equation on $\mathbb{H}^n$}
Given that our goal is to extend the mild solution theory presented by Kato in \cite{KatoMild} to the hyperbolic space $\mathbb{H}^n$, we must first discuss the proper formulation of \eqref{NSRn} on a complete Riemannian manifold $(M,g)$. Specifically, we must make a decision concerning how to generalize the Laplacian $\Delta_{\R^n}$ to a Riemannian manifold, since in this context there is no canonical choice. 

For instance, throughout the course of this work, we will consider the following operators: the Laplace-Beltrami operator on functions defined by
\[
\Delta_g f= \text{div}(\text{grad }f) = \frac{1}{\sqrt{|g|}} \frac{\partial}{\partial x^j}(\sqrt{|g|}g^{ij} \frac{\partial f}{\partial x^i}),
\]
where $|g|$ denotes the determinant of the metric $g$; the Bochner Laplacian on the space of rank $(k,0)$ tensor fields, $\mathcal{T}^k_0(M)$, defined by
\[
\Delta_{B,k} u = - \nabla^* \, \nabla u \quad \text{for } k = 0, 1, 2 \ldots,
\]
where $\nabla$ is the induced Levi-Civita connection on $\mathcal{T}^k_0(M)$; and the Hodge Laplacian acting on the space of differential $k$-forms $\Omega^k(M)$ defined by 
\[
\Delta_{H,k} u = (d^*d + dd^*) u \quad \text{for } k = 0, 1, 2 \ldots, n.
\]
We remark here that with the sign conventions adopted above, $\Delta_g$ and $\Delta_{B,k}$ are negative operators and $\Delta_{H,k}$ is positive.

For functions, one can show all of the above Laplacians coincide (up to a plus or minus sign). However, though $\Omega^k(M) \subset \mathcal{T}^k_0(M) $, the Bochner and Hodge Laplacians do not in general agree for $k$-forms with $k\geq 1$ and are instead related by the Bochner-Weitzenb{\"o}ck identity (see \cite{jost}, for example). For our purposes, the 1-form version is sufficient, which is given by
\begin{equation} \label{wb}
\Delta_{H,1} u = -{\Delta}_{B,1} u +\text{ Ric } u,
\end{equation}
where $\text{Ric}:T^*M \to T^*M$ is the Ricci operator on 1-forms.

Following Ebin-Marsden \cite{ebin_mars}, we use yet another operator
\begin{equation} \label{hodgeL}
L = -2\text{Def}^*\text{Def}  = -\Delta_{H,1} - dd^* + 2\text{Ric},
\end{equation}
where $\text{Def}$ is the deformation tensor defined by $\text{Def } u =\frac{1}{2}\big(\nabla u + (\nabla u)^T\big)$, and $\text{Def}^*$ is its adjoint. 

On $\mathbb{H}^n$ with $d^*u=0$,
\[
L u = \Delta_{B,1}u -(n-1)u,
\]
thus the Navier-Stokes equation becomes
\begin{equation} \tag{$\text{N-S}_{\mathbb{H}^n}$} \label{NSH}
\begin{split}
\partial_t u -\Delta_{B,1}u + (n-1)u+\nabla_{u^{\#}} u + d p &= 0 \\
d^* \, u &= 0, \quad u(0,t) = a,
\end{split}
\end{equation}
where $u$ is now a differential 1-form on $M,$ since $L$ as defined sends 1-forms to 1-forms. As well, $\#$ represents the musical isomorphism $(\cdot)^{\#}: T^*M \to TM$. See \cite{Chan-C-Disconzi} for more on the formulation of the Navier-Stokes equation on Riemannian manifolds.

As in \cite{Pierfelice} and in the spirit of equation \eqref{NSRn'}, we can write equation \eqref{NSH} in pressure-free form. To do this, we take $d^*$ of \eqref{NSH} and use the fact that $d^*u = 0$ for a smooth solution $(u,p)$ of \eqref{NSH}, which gives
\begin{equation} \label{PEllip}
-\Delta_g p + d^* \nabla_{u^{\#}} u = 0.
\end{equation}
Thus 
\begin{equation}
dp =  -d\,(-\Delta_g)^{-1} d^*\, \nabla_{u^{\#}} u.
\end{equation}
Therefore, we define an operator on $T^*\mathbb{H}^n$ by 
\begin{equation} \label{projdef}
\mathbb{P} = I - d\,(-\Delta_g)^{-1} d^*,
\end{equation}
which allows us to rewrite \eqref{NSH} as follows: 
\begin{equation} \tag{$\text{N-S}_{\mathbb{H}^{n}}'$} \label{NSH'}
\begin{split}
\partial_t u -Lu + \mathbb{P }\, (\text{div}(u^{\#} \otimes u^{\#}))^{\flat} &=0  \\
d^* \, u &= 0,  \quad u(0,t) = a,
\end{split}
\end{equation}
where we have also used the fact that $\nabla_{u^{\#}} u = (\text{div}(u^{\#} \otimes u^{\#}))^{\flat}$ when $d^* u =0$. Here, $\flat$ represents the musical isomorphism $(\cdot)^{\flat}: T\mathbb{H}^n \to T^*\mathbb{H}^n$. 

Also, as in the case of $\R^n$, we can then convert \eqref{NSH'} into the following integral equation
\begin{equation} \label{INT}
u=u_0(t)+Gu(t),
\end{equation}
where
\begin{equation} \tag{$\text{N-S}_{\mathbb{H}^{n},\text{int}}'$}  \label{inteq}
u_0(t) = e^{tL}a, \quad Gu(t) = -\int_0^te^{(t-s)L}\mathbb{P }\,\big(\text{div}(u^{\#} \otimes u^{\#}))^{\flat}\big)(s)\,ds, \quad d^*a = 0.
\end{equation}
Note that in defining \eqref{inteq}, we have used Strichartz' result in \cite{Strichartz_laplacian} that $\Delta_{B,1}$ is self-adjoint on a complete Riemannian manifold $(M,g)$, so that the semigroup $e^{tL}$ is well defined. In subsequent sections, we will focus our attention on this integral formulation of the Navier-Stokes equation on $\mathbb{H}^n$.

\subsection{Overview of previous results}
First we will overview some classical results concerning mild solutions to Navier-Stokes on $\R^n$. Then we will examine some known results for mild solutions to Navier-Stokes in the setting of a complete Riemannian manifold $(M,g)$.

As mentioned, the theory of mild solutions to the Navier-Stokes equation goes back to both 1959 and the results of Sobolevski\u{\i} \cite{sobolevskii} and 1961 and the work of Fujita and Kato \cite{kato_fujita}. These works showed local existence and uniqueness results for the Navier-Stokes equation on bounded domains in $\R^2$ for initial data in $L^2$, and in $\R^3$ for sufficiently regular initial data.

Also in $\R^3$ and for initial data in $L^p$, $p>3$, existence of local solutions in certain $L^sL^q$ spaces was shown for the full space $\R^3$ by Fabes, Jones, and Rivi{\`e}re \cite{fabes-jones-riviere} in 1972, by Lewis \cite{Lew} for $\R^3_{+}$ in 1973, and by Fabes, Lewis, and Rivi{\`e}re \cite{flr} for bounded domains in 1977. For more on the case $\R^3$ and $L^p$ with $p>3$, see also von Wahl \cite{von1980regularity} (1980), Miyakawa \cite{miyakawa1981} (1981), and Giga \cite{giga_semilinear} (1986).

In 1980, Weissler \cite{weissler} constructed local solutions for initial data in $L^3(\R_{+}^3)$. Then in 1984, Kato \cite{KatoMild} presented a theory for mild solutions to Navier-Stokes with initial data in $L^n(\R^n)$, $2 \leq n < \infty$, and initial data in $L^n(\R^n) \cap L^p(\R^n)$ for $1 < p < n$. In this work, Kato also showed global well-posedness and $L^n$ and $L^p$-norm time decay properties in the case of small initial data.

Related to Kato's 1984 paper is a 1983 preprint by Giga that was cited by Kato in \cite{KatoMild} and eventually published in 1986 as \cite{giga_semilinear}, where Giga considered mild solutions to semilinear parabolic equations of the form
\[
u_t + Au=Fu,
\]
where $A$ is an elliptic operator and $Fu$ represents the nonlinearity of the equation. In this work, Giga proved well-posedness results in $L^qL^p$ spaces, where $q$ and $p$ are chosen so that the $L^qL^p$ norm is either dimensionless or scaling invariant. Kato then used this to prove a decay result for the $L^n(\R^n)$ norm of a mild solution to Navier-Stokes as part of the aforementioned theory presented in \cite{KatoMild}.

In 1985, Giga and Miyakawa extended the results of Fujita and Kato \cite{kato_fujita} on bounded domains in $\R^2$ and $\R^3$ from $L^2$ to a full $L^p$ theory for $1 < p < \infty$, while also getting rid of the regularity requirement for the initial condition mentioned above for $\R^3$. The case $L^3(D)$, where $D \subset \R^3$ is an exterior domain, was established by Iwashita \cite{iwashita} in 1989.

Since much of this current work is focused on global decay properties for mild solutions, we also survey some results in this direction. As discussed, Kato in his 1984 paper \cite{KatoMild} showed time decay properties for mild solutions with small initial data coming from $L^n(\R^n)$. Specifically, he showed the $L^2(\R^2)$ and $L^2(\R^3)$ norms of global solutions coming from small data decay to zero. Kato did not answer whether the $L^n(\R^n)$ decays to zero for $n\geq 3$ (for more on this, see Section \ref{cmpcnt}). 

Later in 1984, Masuda \cite{masuda1984weak} showed the $L^2(\R^3)$ norm of a Leray-Hopf weak solution decays to zero, though he did not show a uniform decay rate. Then in 1985, Schonbek \cite{schonbek19852} achieved algebraic decay rates for the $L^2(\R^3)$ norm of Leray-Hopf weak solutions coming from initial data in $L^1(\R^3) \cap L^2(\R^3)$ by using Fourier splitting methods. Many years later, in 2003, Gallagher, Iftimie, and Planchon \cite{gall} answered a question left open by Kato \cite{KatoMild} concerning the decay of the $L^3(\R^3)$ norm of a mild solution coming from small data. In fact, these authors show something stronger by proving that any a priori global solution to Navier-Stokes that is continuous in time must decay to zero. Then in \cite{auscher2004stability}, Auscher, Dubois, and Tchamitchian proved global solutions coming from small initial data in the closure of the Schwarz class in $BMO^{-1}$ decay to zero, which is optimal in the sense that it encompasses other known stability results.

For further information on mild solutions on $\R^n$, see also \cite{buckmaster, fabes-jones-riviere, Furioli2000, kuka, lem-rie, lem-rie2,  lions, tsai-lec}.

Having surveyed some of the results concerning mild solutions to Navier-Stokes on Euclidean space, we now collect together the well-posedness and time decay results of Kato for $\R^n$ from \cite{KatoMild}, which we will subsequently extend to $\mathbb{H}^n$.

\begin{thm} \cite{giga_semilinear, KatoMild}  \label{84thm}
Let $a \in L^n(\R^n)$ with $\text{div } a = 0$ and $n \geq 2$. 
\begin{enumerate}
\item[($K1$)] Then there is $T>0$ and a unique solution $u$ of \eqref{NSRn} such that 
\begin{align} 
\label{mainspace2} t^{\big(\frac{1}{2} -\frac{n}{2q} \big)} u &\in BC\big([0,T), L^q(\R^n)\big) \quad \text{for } \quad n \leq q \leq \infty, \\
\label{mainDspace2} t^{\big(1 -\frac{n}{2q} \big)} \nabla u &\in BC\big([0,T), L^q(\R^n)\big) \quad \text{for } \quad n \leq q < \infty,
\end{align} 
with values zero at $t=0$, except for $q=n$ in \eqref{mainspace2}, in which case $u(0) = a$. As well, $u$ has the additional property
\begin{equation} \label{gigares}
u \in L^r((0,T_1), L^q(\R^n)) \quad \text{with } \frac{1}{r} = \frac{1}{2} - \frac{n}{2q}, \,\, n < q < \frac{n^2}{n-2},
\end{equation}
for some $0 < T_1 \leq T$.

\item[($K2$)] There is $\lambda >0$ such that if $\|a\|_{L^n(\R^n)}\leq \lambda$, then the solution from $(K1)$ is global and we may take $T=T_1=\infty$.
\item[($K2'$)] In the situation $(K2)$ where the solution $u$ is global, we have
\begin{equation} \label{intnormdcy}
\lim_{T \to \infty} \frac{1}{T}\int_0^T \|u(t)\|_{L^n(\R^n)}\, dt = 0.
\end{equation}

\item[($K3$)] If $a \in L^p(\R^n)\cap L^n(\R^n)$, where $1 < p < n$, then the solution given by $(K1)$ has the following additional properties.
\begin{align} 
\label{mainLpspace2} u &\in BC\big([0,T_2), L^p(\R^n)\cap L^n(\R^n)\big), \\
\label{mainLpDspace2} t^{\frac{1}{2}} \nabla u &\in BC\big([0,T_2), L^p(\R^n)\cap L^n(\R^n)\big), 
\end{align} 
for some $0 < T_2 \leq T$.
\item[($K4$)] There is some $0 < \lambda_1 \leq \lambda$ such that if $\|a\|_{L^n(\R^n)}\leq \lambda_1$, then the solution in $(K3)$ is global and we may set $T=T_1=T_2=\infty$. Moreover, for any finite $q \geq p$,
\begin{align} 
\label{mainLpspace3} t^{\big(\frac{n}{2p} - \frac{n}{2q}\big)}u &\in BC\big([1,\infty), L^q(\R^n)), \\
\label{mainLpDspace3} t^{\big(\frac{n}{2p} - \frac{n}{2q}+\frac{1}{2}\big)} \nabla u &\in BC\big([1,\infty), L^q(\R^n)\big), 
\end{align} 
provided that the exponent of $t$ is smaller than 1 (separately for each of $u$ and $\nabla u$), otherwise the exponent should be replaced by an abritrary number smaller than 1, but close to 1.
\item[($K4'$)] In the situation $(K4)$ where the solution $u$ is global, we have
\begin{equation}
\lim_{t\to \infty} \|u(t)\|_{L^p(\R^n)} = 0.
\end{equation}
\end{enumerate}
\end{thm}

\begin{remark}
The property \eqref{gigares} in $(K1)$ for the free solution $u_0$ is owed to Giga, who proved the result in \cite{giga_semilinear}. Kato then extended the result to the full solution $u$ and used this to prove the norm decay result \eqref{intnormdcy} in $(K2')$.
\end{remark}

\begin{remark} The result $(K2')$, combined with the energy inequality, which holds for any smooth solution $(u,p)$ of \eqref{NSRn} by the structure of the equation, can be used to show
\begin{equation} \label{L2normdecay}
\lim_{t \to \infty} \|u(t)\|_{L^2(\R^2)} = 0,
\end{equation}
which answered an open question of Leray's in  \cite{LerayExt} and \cite{leray1934} concerning decay of the $L^2$ norm for $\R^2$, at least for the case of mild solutions and small initial data. 

But as Kato mentions in \cite{KatoMild}, he was not able to answer whether
\begin{equation} \label{Lnnormdecay}
\lim_{t \to \infty} \|u(t)\|_{L^n(\R^n)} = 0,
\end{equation}
for $n\geq 3$. As mentioned previously, this was answered in the affirmative in 2003 by Gallagher, Iftimie, and Planchon \cite{gall} for $n=3$ (see also, Tsai \cite{tsai-lec}). 
\end{remark}

\begin{remark} Leray's question concerning decay of the $L^2$ norm for solutions in $\R^n$ for $n\geq 3$ is answered, at least for mild solutions and small initial data, by Kato's result $(K4')$. In particular, we have
\begin{equation} \label{LpRndcy}
\lim_{t \to \infty} \|u(t)\|_{L^2(\R^3)} = 0.
\end{equation}
\end{remark}

Moving now to the setting of a complete Riemannian manifold, we first mention the work of Mitrea and Taylor in \cite{MitreaTaylor}, where they considered the inhomogenous form of \eqref{NSH} in the setting of Lipschitz subdomains of compact Riemannian manifolds and proved local existence and uniqueness results for mild solutions (see also \cite{michael1999partial}).

Then in \cite{Pierfelice}, Pierfelice considered the Navier-Stokes equation on a class of non-compact and complete Riemannian manifolds with positive injectivity radius and satisfying certain curvature bounds (see \cite{Pierfelice} for details). In this setting, Pierfelice proved dispersive and smoothing estimates for the semigroup $e^{tL}$ that both parallel and improve upon known dispersive and smoothing estimates for the heat semigroup on $\R^n$ (see \cite{KatoMild}, for example). These estimates then allowed Pierfelice to extend Kato's well-posedness result $(K1)$, specifically \eqref{mainspace2}, and the global in time result $(K2)$ for small initial data to the setting of complete Riemannian manifolds where the Ricci operator is a negative constant scalar multiple of the metric. We state the result as follows, where $c_n(t) = C(n)\text{max }(t^{-n/2},1)$ for some constant $C(n)$ depending only on the dimension.

\begin{thm}\cite{Pierfelice}\label{PFthm}
Let $M$ be a complete Riemannian manifold where the Ricci operator is a negative constant scalar multiple of the metric $g$, and let $a \in L^n(M)$ with $d^* a =0$. Then there exists $T>0$ and a unique solution $u$ of Navier-Stokes on $M$ such that
\begin{equation} \label{mainspace1}
\begin{split}
u &\in BC\big([0,T), L^n(M)\big) \\
c_n(t)^{-\big(\frac{1}{n} -\frac{1}{q} \big)}e^{t\beta} u &\in BC\big([0,T), L^q(M)\big), \quad n < q < \infty,
\end{split}
\end{equation} 
where $\beta >0$ depends on $n$ and $q$ and such that we have continuous dependence on the initial data. Moreover, there is $\lambda > 0$ such that if $\|a\|_{L^n(\mathbb{H}^n)} \leq \lambda$, then the solutions are global in time. 
\end{thm}

Examining \eqref{mainspace2} and \eqref{mainspace1}, we see that for $0 \leq t \leq 1$, Pierfelice's result for $\mathbb{H}^n$ coincides with Kato's \eqref{mainspace2} in Theorem \ref{84thm}. However, for the large time regime $t >1$, Pierfelice's result offers an improvement on Kato's for the case $q>n$, giving exponential decay in time as opposed to decay according to an inverse power of $t$. We also note that Pierfelice's result does not include $q=\infty$, whereas Kato's does. 

In fact, Pierfelice does more in \cite{Pierfelice} and extends Theorem \ref{PFthm} to arbitrary non-compact and complete Riemannian manifolds with positive injectivity radius and satisfying certain more general curvature properties (see \cite{Pierfelice} for more information). In this case, it is required that the initial data also satisfy
\begin{equation}
a \in L^n(M) \cap L^2(M).
\end{equation}
As well, for global existence in this scenario, both the $L^n$ and the $L^2$ norms are required to be small, rather than just the $L^n$ norm. 

In what follows, we will show that Pierfelice's dispersive and smoothing estimates proved for the semigroup $e^{tL}$ on a specific class of Riemannian manifolds allow extension of Kato's Theorem \ref{84thm} to the setting of $\mathbb{H}^n$, including all results about the total covariant derivative $\nabla u$ and all global in time decay results. Additionally, we will show that in extending the results $(K1)$, $(K3)$, and $(K4)$ of Theorem \ref{84thm}, it is possible to achieve both Kato's original power of $t$-type decay, and exponential time decay in the spirit of Pierfelice's Theorem \ref{PFthm}. 

This exponential decay is an improvement over the Euclidean theory presented by Kato in \cite{KatoMild}, and has the added benefit of greatly simplifying the proofs of the extensions of results $(K2')$ and $(K4')$ from Theorem \ref{84thm} to the setting of $\mathbb{H}^n$. The exponential decay also offers a wider range of space-time $L^rL^q$ spaces than those specified by \eqref{gigares} of Theorem \ref{84thm}, including the cases $q =n$, and for the situation when $a \in L^n(\H^n) \cap L^p(\H^n)$, $q=p$, which are not covered by Kato's results in \cite{KatoMild}. 

As well, the exponential time decay allows us to show
\begin{equation} \label{Lnnormdecay}
\lim_{t \to \infty} \|u(t)\|_{L^n(\mathbb{H}^n)} = 0, \quad n \geq 2
\end{equation}
which, as discussed above, was a result Kato was unable to show in \cite{KatoMild} on $\R^n$ for $n\geq 3$. However, as is the case for Pierfelice in \cite{Pierfelice}, we do not get the case $q=\infty$ in \eqref{mainspace2} and \eqref{mainDspace2}.


\subsection{Main results}
The following theorems collect together the statements to be proven. 

\begin{thm}\label{EUthm}
Let $a \in L^n(\mathbb{H}^n)$, $n\geq 2$, with $d^* a =0$. Then there exists $T>0$ and a unique solution $u$ of \eqref{NSH'} such that
\begin{align}
 \label{mainspace} t^{\big(\frac{1}{2} -\frac{n}{2q} \big)}e^{t\beta } u &\in BC\big([0,T), L^q(\mathbb{H}^n)\big), \quad n \leq q < \infty, \\
 \label{mainDspace} t^{1-\frac{n}{2q}}e^{t\beta'' }\nabla u  &\in BC\big([0,T), L^q(\mathbb{H}^n)\big),  \quad n \leq q < \infty, 
\end{align}
both with values zero at $t=0$ except for $q = n$ in \eqref{mainspace}, in which case $u(0) = a$, and where $\beta >0$ appearing in \eqref{mainspace} depends on $n$, $q$, and any fixed $0< \delta < 1$, and where $\beta''>0$ appearing in \eqref{mainDspace} depends only on $n$, $q$, and a fixed $0 < \delta < 1$ chosen so that $q<\frac{n}{1-\delta}$. Additionally, we have continuous dependence on the initial data, and there is $\lambda > 0$ such that if $\|a\|_{L^n(\mathbb{H}^n)} \leq \lambda$, then the solution is global in time. 
\end{thm}

\begin{remark}
This is similar to Pierfelice's Theorem \ref{PFthm} given above for the case of $\mathbb{H}^n$, though now expanded to include information about $\nabla u$. As well, this result has exponential decay also for the $L^q(\H^n)$ norm of a solution for $n\leq q < \infty$, whereas Pierfelice's Theorem \ref{PFthm} only gives this for $n < q <\infty$. Note also that, rather than involving the aforementioned function
\[
c_n(t) = C(n)\max\{t^{-n/2,1}\},
\]
the results \eqref{mainspace} and \eqref{mainDspace} in Theorem \ref{EUthm} combine the power of $t$-type time decay from $(K1)$ in Kato's \ref{84thm} with exponential time decay, and this is valid for the full interval of existence $[0,T)$, with no split happening at $t=1$. As a result, the constants $\beta$ appearing in the exponential terms in \eqref{mainspace} and \eqref{mainDspace} from Theorem \ref{EUthm} are possibly different than the constant appearing in the exponential term from Pierfelice's Theorem \ref{PFthm}. For more details on this, see Section \ref{estimatesused}.  

Also, as mentioned above, unlike the case of Kato's Theorem \ref{84thm}, we do not have $q=\infty$ in Theorem \ref{EUthm}, though we hope to study this case further in future works.
\end{remark}

\begin{thm}\label{PSL}
For the solution $u$ and time of existence $T>0$ from Theorem \ref{EUthm}, there exists some $0 < T_1 \leq T$ such that
\begin{equation} \label{LrLq}
u \in L^{r}\big((0,T_1), L^q(\mathbb{H}^n)\big) \quad\quad \text{with } \quad \frac{1}{2} - \frac{n}{2q} = \frac{1}{r}, \quad n < q < \frac{n^2}{n-2}.
\end{equation}
As well, there is $\lambda_1>0$ such that if $\|a\|_{L^n(\mathbb{H}^n)} \leq \lambda_1$, then $T_1$ can be extended to $+\infty$.
\end{thm}

\begin{remark}
This result is the same as in the Euclidean setting and as we will see in the course of proving Theorem \ref{PSL}, the requirements on $r$ and $q$ listed in \eqref{LrLq} are a consequence of the Marcinkiewicz interpolation theorem. We also mention here that, as noted by Stein in Appendix B of \cite{Stein}, the Marcinkiewicz interpolation theorem is valid if the underlying measure space $\R^n$ of $L^p(\R^n)$ is replaced by a more general measure space, such as $\mathbb{H}^n$. 

We note here that we do not currently have exponential time decay for Theorem \ref{PSL}. However, the exponential decay present in \eqref{mainspace} of Theorem \ref{EUthm} allows us to expand the range of $L^rL^q$ spaces relative to those specified in Theorem \ref{PSL}, which is the subject of the next two theorems. 
\end{remark}

\begin{thm} \label{newLrLq}
For the solution $u$ and time of existence $T>0$ from Theorem \ref{EUthm},
\be \label{newPSL}
u \in L^{r}\big((0,T), L^q(\mathbb{H}^n)\big),
\ee
where $1 \leq r < \infty$ and $n < q$ satisfy the inequality
\be \label{newPSL'}
\frac{1}{2} - \frac{n}{2q} < \frac{1}{r}.
\ee
Moreover, if $\|a\|_{L^n(\mathbb{H}^n)} \leq \lambda$, where $\lambda$ is the constant from Theorem \ref{EUthm}, then this estimate holds for $T= +\infty$.
\end{thm}

\begin{remark} \label{lrlqrem1}
We first observe that while this theorem is true locally on $\R^n$, to the best of our knowledge, it has not been shown globally, and such a global result is not achievable by the estimates used in Kato's theory as outlined in \cite{KatoMild}. For more details on this, see Section \ref{cmpcnt}.

Next we note that this theorem asserts that the $L^q$ norm of a solution $u$ can be placed in any $L^r$ space in time for $r \in [1,\infty)$ by choosing a suitable $n < q$, and that this can be done globally. In contrast, taking $n=3$ in Theorem \ref{PSL}, it follows that
\be
r = \frac{2q}{q-3}, \quad 3 < q < 9,
\ee
which only gives a possible range for $r$ of $(3, \infty)$. A similar argument shows for $n=2$ that $r \in (2, \infty)$. 

Since only an analog Theorem \ref{PSL} is known to hold globally on $\R^n$, and not an analog Theorem \ref{newLrLq}, the preceding discussion shows that, when considering a global mild solution coming from small $L^n(\H^n)$ data, the range of $L^r$ spaces in time available on $\H^n$ is larger than that available on $\R^n$.
\end{remark}

\begin{thm} \label{LrLn} 
For the solution $u$ and time of existence $T>0$ from Theorem \ref{EUthm},
\be
u \in L^{r}\big((0,T), L^n(\mathbb{H}^n)\big) \quad\quad \text{for } \quad 1 \leq r < \infty.
\ee
Moreover, if $\|a\|_{L^n(\mathbb{H}^n)} \leq \lambda$, where $\lambda$ is the constant from Theorem \ref{EUthm}, then this estimate holds for $T= +\infty$.
\end{thm}

\begin{remark} An analog of this result is not given by Kato for $\R^n$ in \cite{KatoMild}, and to the best of our knowledge, as is the case for Theorem \ref{newLrLq}, this result only seems possible locally on $\R^n$. For more information on this, see Section \eqref{cmpcnt}.

\end{remark}

\begin{thm}\label{tmdcy1}
For the case in Theorem \ref{EUthm} when the solution $u$ is global, we have
\begin{equation} \label{tmint}
\lim_{T \to \infty} \frac{1}{T} \int_0^T \|u(t)\|_{L^n(\mathbb{H}^n)} \,dt = 0,
\end{equation}
and
\begin{equation} \label{Lndecay}
\lim_{t \to \infty} \|u(t)\|_{L^n(\mathbb{H}^n)} = 0,
\end{equation}
for all $n \geq 2$.
\end{thm}

\begin{remark}
As discussed above, the main purpose for Kato in proving 
\begin{equation}
\lim_{T \to \infty} \frac{1}{T} \int_0^T \|u(t)\|_{L^n(\mathbb{R}^n)} \,dt = 0,
\end{equation}
was that it allowed him to show
\begin{equation} \label{limques}
\lim_{ t \to \infty} \|u(t)\|_{L^n(\mathbb{R}^n)} =0.
\end{equation}
for the specific case $n=2$. However, Kato was not able to answer whether \eqref{limques} holds on $\R^n$ for $n\geq 3$, and it took some time until this was answered in the affirmative by Gallagher, Iftimie, and Planchon \cite{gall}, as previously mentioned, and as discussed in further detail below.
 
In the setting of $\mathbb{H}^n$, Theorem \ref{tmdcy1} implies, with the same argument as outlined by Kato in \cite{KatoMild} and discussed above, that
\begin{equation}
\lim_{ t \to \infty} \|u(t)\|_{L^2(\mathbb{H}^2)} =0.
\end{equation}
We also note that, though the result \eqref{tmint} is the same as $(K2')$ from Kato's Theorem \ref{84thm}, the proof in the setting of $\mathbb{H}^n$ is greatly simplified due to the presence of the exponential term appearing in \eqref{mainspace} from Theorem \ref{EUthm}, which we will discuss in further detail in Section \ref{cmpcnt}.

Moreover, though we include result \eqref{tmint} to show that Kato's theory in \cite{KatoMild} can be extended to $\mathbb{H}^n$, we can now directly show 
\[
\lim_{t \to \infty} \|u(t)\|_{L^n(\mathbb{H}^n)} = 0,
\]
for all $n\geq 2$, without referencing the result \eqref{tmint} at all. As well, the method of proof for this result does not depend on the dimension or the energy inequality, and uses only the exponential decay appearing in \eqref{mainspace2} from $(K1)$, which is considerably different than the situation on Euclidean space, as we discuss in further detail in Section \ref{cmpcnt}.

Also in this direction, we show explicitly that the convergence in \eqref{Lndecay} is of order $e^{-t\beta}$, where $\beta$ comes from \eqref{mainspace}. 
\end{remark}

\begin{thm} \label{Lpthm}
Let $a \in L^p (\mathbb{H}^n) \cap L^n(\mathbb{H}^n)$ with $d^* a =0$ and where $1 < p < n$, $n\geq 2$. Then the solution given by Theorem \ref{EUthm} has the following additional properties:
\begin{align}
\label{lpspace} e^{t\beta^* } u &\in BC\big([0,T), L^p(\mathbb{H}^n) \cap L^n(\mathbb{H}^n)\big), \\
\label{lpDspace} t^{1/2}e^{t\beta^{**} } \nabla u &\in BC\big([0,T), L^p(\mathbb{H}^n) \cap L^n(\mathbb{H}^n)\big),
\end{align}
where the possible different constants $\beta^* >0$ and $\beta^{**}>0$ appearing in \eqref{lpspace} and \eqref{lpDspace} depend only on $n$, $p$, and $0<\delta<1$ chosen so that $n/p+\delta < n$. Moreover, we have continuous dependence on the initial data, and if $\|a\|_{L^n(\mathbb{H}^n)} \leq \lambda$, where $\lambda$ is the same constant appearing in Theorem \ref{EUthm}, then the solution is global in time. 
\end{thm}

\begin{remark}
Though $\|a\|_{L^n(\mathbb{H}^n)}$ is required to be small for global existence and uniqueness in Theorem \ref{Lpthm}, there is no restriction on the size of $\|a\|_{L^p(\mathbb{H}^n)}$. This is the same as the situation for $\R^n$.

However, in contrast to the situation for $\R^n$ and results \eqref{mainLpspace2} and \eqref{mainLpDspace2} from Kato's Theorem \ref{84thm}, we have additional exponential time decay in \eqref{lpspace} and \eqref{lpDspace}.
\end{remark}

\begin{thm} \label{LrLp}
In the situation described by Theorem \ref{Lpthm}, the solution $u$ also satisfies
\be
u \in L^{r}\big((0,T), L^p(\mathbb{H}^n) \cap L^n(\H^n)\big) \quad\quad \text{for } \quad 1 \leq r < \infty.
\ee
Moreover, if $\|a\|_{L^n(\mathbb{H}^n)} \leq \lambda$, where $\lambda$ is the constant from Theorem \ref{EUthm}, then this estimate holds for $T= +\infty$.
\end{thm}

\begin{remark} As is the case for Theorems \ref{newLrLq} and \ref{LrLn}, an analog of Theorem \ref{LrLp} is not given by Kato for $\R^n$ in \cite{KatoMild}, and only seems possible locally on $\R^n$. For more information on this, see Section \eqref{cmpcnt}.

\end{remark}
 
 \begin{thm}\label{tmdcy2}
For the case in Theorem \ref{Lpthm} when the solution $u$ and its derivative $\nabla u$ are global, the following decay estimates on $u$ hold for any $1<p \leq q < \infty$:
\begin{equation} \label{td3}
t^{\big(\frac{n}{2p}-\frac{n}{2q}\big)}e^{t\tilde{\beta} } u \in BC([1,\infty), L^q(\mathbb{H}^n)) \quad\quad \text{if} \quad \frac{n}{2p}-\frac{n}{2q} < 1, 
\end{equation}
\begin{equation} \label{td4}
t^{\big(\frac{n}{2p'}-\frac{n}{2q}\big)}e^{t\tilde{\tilde{\beta}} } u \in BC([1,\infty), L^q(\mathbb{H}^n)) \quad\quad \text{if} \quad \frac{n}{2p}-\frac{n}{2q} \geq 1,
\end{equation}
where $p'$ is chosen such that $p < p' < n$, $p' \leq q$, and $\dfrac{n}{2p'}-\dfrac{n}{2q} < 1$, but such that $\dfrac{n}{2p'}-\dfrac{n}{2q}$ is arbitrarily close to 1. The value of $\tilde{\beta} >0$ appearing in the exponential term in \eqref{td3} depends only on $n$, $p$, $q$, and fixed constants $ \delta, \delta', \delta^* \in (0,1)$ satisfying
\begin{equation} \label{3deltas}
\begin{split}
\frac{n}{p} - \frac{n}{q} +\delta &< 2, \\
1-\delta ' &< \delta, \\
\frac{n}{p}+\delta^* &< n,
\end{split}
\end{equation}
and the value of the constant $\tilde{\tilde{\beta}}$ appearing in the exponential term in \eqref{td4} depends only on $n$, $p'$, $q$, and the same choices of $ \delta, \delta', \delta^* \in (0,1)$ as in \eqref{3deltas}, only with $p'$ replacing $p$.

Furthermore, we have the following decay estimates on $\nabla u$ for any $1<p \leq q < \infty$:
\begin{equation} \label{td5}
t^{\big(\frac{n}{2p}-\frac{n}{2q}+1\big)}e^{t\tilde{\beta} } \nabla u \in BC([1,\infty), L^q(\mathbb{H}^n)) \quad\quad \text{if} \quad \frac{n}{2p}-\frac{n}{2q} +\frac{1}{2}< 1, 
\end{equation}
\begin{equation} \label{td6}
t^{\big(\frac{n}{2p'}-\frac{n}{2q}+1\big)}e^{t\tilde{\tilde{\beta}}} \nabla u \in BC([1,\infty), L^q(\mathbb{H}^n)) \quad\quad \text{if} \quad \frac{n}{2p}-\frac{n}{2q}+\frac{1}{2} \geq 1,
\end{equation}
where $p'$ is chosen such that $p < p' < n$, $p' \leq q$, and $\dfrac{n}{2p'}-\dfrac{n}{2q} +\dfrac{1}{2} < 1$, but such that $\dfrac{n}{2p'}-\dfrac{n}{2q}+\dfrac{1}{2}$ is arbitrarily close to 1. The value of $\tilde{\beta} >0$ appearing in the exponential term in \eqref{td5} depends only on $n$, $p$, $q$, and fixed constants $ \delta, \delta', \delta^*\in(0,1)$ satisfying
\begin{equation} \label{3deltas1}
\begin{split}
\frac{n}{p} - \frac{n}{q} +1+\delta &< 2 \\
1-\delta ' &< \delta \\
\frac{n}{p}+\delta^* &< n,
\end{split}
\end{equation}
and the value of the constant $\tilde{\tilde{\beta}}$ appearing in the exponential term in \eqref{td6} depends only on $n$, $p'$, $q$, and the same choices of $ \delta, \delta', \delta^* \in (0,1)$ as in \eqref{3deltas1}, only with $p'$ replacing $p$.
\end{thm}

\begin{remark}
This is the analog on $\mathbb{H}^n$ of Kato's $(K4)$ from Theorem \ref{84thm}, though here, we have additional exponential time decay present in results \eqref{td3}, \eqref{td4}, \eqref{td5}, and \eqref{td6}. 
\end{remark}

\begin{remark}
Theorem \ref{tmdcy2} is written in terms of the time interval $[1, \infty)$, as this was how the result was presented by Kato in \cite{KatoMild}. However, examination of the proof shows the results in Theorem \ref{tmdcy2} to be true for $[0, \infty)$ as well, and this is the case for both $\R^n$ and $\H^n$.
\end{remark}

\begin{thm}\label{Lptimedecay} 
For the case in Theorem \ref{Lpthm} when the $L^p$ solution $u$ and its derivative $\nabla u$ are global, 
\begin{equation} \label{td2}
\lim_{t \to \infty}  \|u(t)\|_{L^p(\mathbb{H}^n)} = 0.
\end{equation}
\end{thm}

\begin{remark}
Theorem \ref{Lptimedecay} implies, as in the case of $\R^3$, that 
\begin{equation}
\lim_{ t \to \infty} \|u(t)\|_{L^2(\mathbb{H}^3)} =0.
\end{equation}

As well, in the course of proving Theorem \ref{Lptimedecay}, we will explicitly show
\begin{equation}
\|u(t)\|_{L^p(\mathbb{H}^n)} = \mathcal{O}(e^{-\beta t}),
\end{equation}
for some $\beta >0$ depending only on $n$, $p$, and a fixed $0 < \delta < 1$.

In comparison, Kato shows in \cite{KatoMild} that on $\R^n$,
\begin{equation}
\|u(t) - e^{tL}a\|_{L^p(\mathbb{R}^n)} = \mathcal{O}(t^{-\omega/2}),
\end{equation}
where $\omega>0$ is any real number satisfying
 \begin{equation} \label{deltacond}
 \omega < \text{min}\bigg\{ 1, n-\frac{n}{p}, \frac{n}{p} - 1\bigg\}.
 \end{equation}
This implies that the decay rate of $\|u(t)\|_{L^p(\mathbb{R}^n)}$ is at least as fast as the slower of $\|e^{tL}a\|_{L^p(\mathbb{R}^n)}$ and $t^{-\omega/2}$.

So in the setting of $\mathbb{H}^n$, not only does the limit \eqref{td2} converge to 0, it does so at a much faster rate. As well, the additional exponential time decay present in results \eqref{lpspace} and \eqref{lpDspace} from Theorem \ref{Lpthm} allows for significant simplifications in the proof of Theorem \ref{Lptimedecay}, which we will discuss in more detail in the next section.
\end{remark}

\begin{remark}
It is enough to prove the results of Theorems \ref{EUthm} and \ref{Lpthm} for $L^{\infty}([0,T))$ in time. Indeed, the main ingredient in extending from $L^{\infty}([0,T))$ to $BC([0,T))$ in time in the Euclidean setting $\R^n$ is the strong continuity of the heat semigroup $e^{t\Delta_{\R^n}}$ on $L^p(\R^n)$ for $1 \leq p <\infty$ (see, for example, \cite{kato_fujita}, \cite{giga_semilinear}, \cite{KatoMild}, and Chapter 5 of \cite{tsai-lec}). Thus if it is known that the corresponding semigroup $e^{tL}$ in the setting of $\mathbb{H}^n$ is strongly continuous, the extension from $L^{\infty}([0,T))$ to $BC([0,T))$ follows just as it does in the Euclidean case. Therefore, in the Appendix, we prove the semigroup $e^{tL}$ is not only strongly continuous, but contractive on  $L^p(\mathbb{H}^n)$ for $1 \leq p <\infty$.
\end{remark}


\subsection{Comparing and contrasting the cases for $\R^n$ and $\mathbb{H}^n$} \label{cmpcnt}
As we have remarked, in \cite{Pierfelice}, Pierfelice proved dispersive and smoothing estimates for the semigroup $e^{tL}$ that improve upon well-known estimates for the heat semigroup on $\R^n$, and then expanded Kato's results \eqref{mainspace2} and $(K2)$ from Theorem \ref{84thm} to the setting of a class of complete Riemannian manifolds containing $\mathbb{H}^n$, while achieving better global in time behavior in certain instances.

Thus it is our hope that this current work, where we extend the rest of Kato's Theorem \ref{84thm} to the setting of $\mathbb{H}^n$, while also showing improved global in time decay for certain $L^p$ norms, can be taken together with Pierfelice's \cite{Pierfelice} to provide for the hyperbolic space as complete a picture as possible of Kato's $L^p$-based mild solution theory, as presented in \cite{KatoMild}.

As a first point of comparison, we note that for any $\beta >0$ and $t\geq 0$, $e^{-\beta t} \leq 1$. Using this fact, we see that Theorems \ref{EUthm} - \ref{Lptimedecay} exactly recover on $\mathbb{H}^n$ Kato's Theorem \ref{84thm} from \cite{KatoMild}, except for the case $q=\infty$ in results \eqref{mainspace2} and \eqref{mainDspace2}. However, not only do the additional exponential in time decay terms appearing in Theorems \ref{EUthm} and \ref{Lpthm} offer an improvement over the analogous results for $\R^n$, they also allow for remarkable simplifications in proving the global in time decay results from Theorems \ref{tmdcy1} and \ref{Lptimedecay}, as compared to their Euclidean analogs \eqref{intnormdcy} and \eqref{LpRndcy}.

To be more explicit, one of Kato's primary motivations in proving Theorem \ref{84thm} from \cite{KatoMild} was to investigate the global behavior of solutions as $t \to \infty$. Indeed, Kato even explicitly states this in Remark 1.1 from \cite{KatoMild}: ``As is well known, the solution $u$ is smooth for $t>0$. Therefore, the real interest in these theorems are in the behavior of $u$ as $t \to 0$ and, in case $u$ is global, as $t \to \infty$." 

But in examining Kato's result \eqref{mainspace2} from Theorem \ref{84thm}, as well as result \eqref{mainLpspace2}, we see that there is no information given for the decay profile of $\|u(t)\|_{L^n(\mathbb{R}^n)}$ or $\|u(t)\|_{L^p(\mathbb{R}^n)}$, respectively, and so Kato has to work quite hard to prove the decay results \eqref{intnormdcy} and \eqref{LpRndcy}. Indeed, though results \eqref{gigares}, \eqref{mainLpspace3}, and  \eqref{mainLpDspace3} from Kato's Theorem \ref{84thm} are interesting results in their own right, they are also essential tools in Kato's proofs of \eqref{intnormdcy} and \eqref{LpRndcy}. Yet, even with these results in hand, which themselves take no small amount of effort, the proofs of \eqref{intnormdcy} and \eqref{LpRndcy} from Kato's Theorem \ref{84thm} are still considerably involved.

As well, for the case of small initial data and global existence, Kato only shows the following:
\begin{equation}
\lim_{t \to \infty} \|u(t)\|_{L^2(\mathbb{R}^2)} = 0,
\end{equation}
and
\begin{equation}
\lim_{t \to \infty} \|u(t)\|_{L^p(\mathbb{R}^n)} = 0,
\end{equation}
where in the second limit, $1 < p < n$. The question as to whether $\|u(t)\|_{L^n(\mathbb{R}^n)}\to 0$ as $t \to \infty$ for $n\geq 3$ was left unanswered, and was not shown for $n=3$ until 2003 by Gallagher, Iftimie, and Planchon \cite{gall}. Moreover, in both Kato's approach in \cite{KatoMild} and in the approach by Gallagher, Iftimie, and Planchon \cite{gall}, elements from the Leray-Hopf weak solution theory such as local energy estimates are required, and the proofs are quite different for dimensions two and three.

In contrast, due to the exponential terms appearing in results \eqref{mainspace} and \eqref{lpspace} from Theorems \ref{EUthm} and \ref{Lpthm}, respectively, in the setting of $\mathbb{H}^n$, we do have information concerning the time decay of the norms $\|u(t)\|_{L^n(\mathbb{H}^n)}$ and $\|u(t)\|_{L^p(\mathbb{H}^n)}$. Thus the proofs for the decay results in Theorems \ref{tmdcy1} and \ref{Lptimedecay} are straightforward and immediate. We are also able to affirmatively show that 
\begin{equation}
\lim_{t \to \infty} \|u(t)\|_{L^n(\mathbb{H}^n)} = 0,
\end{equation}
for $n\geq 2$, and do so using a unified proof that works for all dimensions. Also, as previously mentioned, the convergence of this limit is of order $e^{-\beta t}$, where $\beta$ comes from \eqref{mainspace2}, so that we have an explicit rate for the decay on $\mathbb{H}^n$. 

The exponential time decay appearing in results \eqref{mainspace} and \eqref{lpspace} from Theorems \ref{EUthm} and \ref{Lpthm}, respectively, also offers a wider range of time integrability for the solution than is available in the Euclidean setting, and this is the focus of Theorems \ref{newLrLq}, \ref{LrLn}, and \ref{LrLp}. As discussed in Remark \ref{lrlqrem1}, for $n=2$, both Kato's result \eqref{gigares} from Theorem \ref{84thm}, and its analog on $\H^2$ given by Theorem \eqref{PSL}, allow us to place the spatial $L^q$ norm of a mild solution $u$ in $L^r(0,T)$ for $2 < r < \infty$ by choosing some $2 < q < \infty$ as defined by \eqref{gigares}. Also, for $n=3$, Kato's Theorem \ref{84thm} and its analog on $\H^3$ given by Theorem \eqref{PSL} give $\|u(t)\|_{L^q} \in L^r(0,T)$ for $3 < r < \infty$ by choosing an appropriate $3 <q < 9$. 

In contrast, Theorem \ref{newLrLq} allows us to put $\|u(t)\|_{L^q(\H^n)} \in L^r(0,T)$ for $1 \leq r < \infty$ by choosing an appropriate $n < q$, and the range of Lebesgue exponents for the time-integrability does not change with dimension. Additionally, Theorem \ref{LrLn} shows that we can achieve $L^r$-integrability in time for $1 \leq r < \infty$ for the $L^n(\H^n)$ norm of the solution, and we also get the range $1 \leq r < \infty$ for the $L^r(0,T)$ norm of $\|u(t)\|_{L^p(\H^n)}$ in the case where the initial data satisfies $a \in L^n(\H^n) \cap L^p(\H^n)$ for $1 < p < n$.

As it turns out, analogs of Theorems \ref{newLrLq}, \ref{LrLn}, and \ref{LrLp} can be shown locally in the Euclidean setting, but as far as we aware, not globally. We first observe for $n < q$ that by Kato's Theorem \ref{84thm}, in the case of a global solutions, there exists a constant $C$ such that
\be
\|u(t)\|_{L^q(\R^n)} \leq Ct^{-\big(\frac{1}{2} -\frac{n}{2q} \big)}, \quad 0 \leq t < \infty.
\ee
So in estimating the $L^r(0,\infty)$ norm of $\|u(t)\|_{L^q(\R^n)}$, we would need to uniformly control the $L^r(0,\infty)$ norm of $t^{-(\frac{1}{2} -\frac{n}{2q})}$. Since the assumptions on $n$, $q$, and $r$ given by \eqref{newPSL'} imply
\be
r\bigg(\frac{1}{2} -\frac{n}{2q}\bigg) < 1,
\ee
we can uniformly control $\int_0^T \|u(t)\|^r_{L^q(\R^n)}\, dt$ locally, say for $0 < T < 1$. However, since the integral 
\be
\int_1^{\infty} t^{-r\big(\frac{1}{2} -\frac{n}{2q} \big)}\,dt
\ee
diverges, it is not clear how to prove an analog of Theorem \ref{newLrLq} globally in the Euclidean setting, since the power of $t$ appearing is not useful for the large time regime.

However, in the case of $\H^n$, we have by result \eqref{mainspace} of Theorem \ref{EUthm} that
\be
\|u(t)\|_{L^q(\H^n)} \leq Ct^{-\big(\frac{1}{2} -\frac{n}{2q} \big)}e^{-t\beta}, \quad 0 \leq t < \infty,
\ee
where $\beta >0$, and it is this additional exponential decay that allows us to uniformly control the $L^r(0,T)$ norm of $\|u(t)\|_{L^q(\H^n)}$, even for large times. Similarly, Kato's results \eqref{mainspace2} and \eqref{mainLpspace3} from Theorem \ref{84thm} give no information concerning a decay rate for $\|u(t)\|_{L^n(\H^n)}$ and $\|u(t)\|_{L^p(\H^n)}$, respectively. Thus the global versions of Theorems \ref{LrLn} and \ref{LrLp} seem out of reach in the Euclidean setting, at least using Kato's theory, since there is no useful time decay available to handle uniform estimates for the large time regime $t \in (1, \infty)$.

We next remark here that, close examination of the proofs of Theorems \ref{EUthm}, \ref{Lpthm}, and \ref{tmdcy2} show that the methods employed in this current work are adaptable to $\R^n$ and offer some improvements. More precisely, since any exponential term of the form $e^{-\beta t}$ is bounded for $\beta >0$ and $t\geq 0$, terms of this form can, if desired, be omitted from the subsequent estimates and proofs appearing in the paper. What remains are proofs in the spirit of those used by Kato in \cite{KatoMild} that are easily adaptable to $\R^n$, though they have been simplified wherever possible.

One such simplifcation involves how to handle the term $Gu$ defined by \eqref{inteq}. The estimates used by Kato in \cite{KatoMild} for this term differ from our estimates to be presented in Section \ref{estimatesused}, even ignoring the presence of exponential terms. Indeed, Kato's estimates involve norms of both $u$ and $\nabla u$. Thus in running his Picard iteration argument for the solution $u$, Kato necessarily requires information about norm bounds and decay rates for both $u$ and $\nabla u$, simultaneously.

In our approach, this is unnecessary,  and by exploiting the fact that
\[
\nabla_{u^{\#}} u = (\text{div}(u^{\#} \otimes u^{\#}))^{\flat} 
\]
whenever $d^* u =0$ and using a general smoothing estimate of Pierfelice stated in \cite{Pierfelice} for the divergence of tensors in $TM \otimes TM$, we are able to derive estimates on the term $Gu$ involving only norms of $u$, and thus are able to prove all of the results for $u$ in Theorems \ref{EUthm} - \ref{Lpthm} without knowing any information about $\nabla u$. 

Since the aforementioned smoothing estimate for the divergence of tensors proved by Pierfelice in \cite{Pierfelice} is easily adaptable to $\R^n$, and in fact, is even used by Kato in \cite{KatoMild} to prove \eqref{gigares} from Theorem \ref{84thm}, it follows that it is similarly possible on $\R^n$ to prove all results concerning the solution $u$ in Kato's $(K1)$ - $(K3)$ from Theorem \ref{84thm}, without requiring any information about the derivative $\nabla u$.

Moreover, because Kato's estimates for $Gu$ intertwine both $u$ and $\nabla u$, Kato has to run concurrent iteration arguments for both $u$ and $\nabla u$ in establishing \eqref{mainspace2} and \eqref{mainDspace}. Our approach is again simplified in that, once \eqref{mainspace} from Theorem \ref{EUthm} is known for finite $q>n$, both the result \eqref{mainspace} for $q=n$ and the derivative result \eqref{mainDspace} for all finite $q \geq n$ can be established by a straightforward application of Gr{\"o}nwall's inequality.

Similarly, for proving $(K3)$ in Theorem \ref{84thm}, Kato again returns to the sequence of Picard iterates and runs convergence arguments for both $u$ and $\nabla u$. In the same way as discussed above, we are able to avoid this by again using \eqref{mainspace} from Theorem \ref{EUthm} for the case $q>n$ together with Gr{\"o}nwall's inequality.

Another benefit of this simplified approach based on Gr{\"o}nwall's inequality is that, whereas the times $T$ in $(K1)$ and $T_1$ in $(K3)$ from Theorem \ref{84thm} may be different, with $T_1 < T$ possible, we are able to prove our Theorems \ref{84thm} and \ref{Lpthm} using the same time of existence $T$. Additionally, while the smallness requirements $\lambda$ in $(K2)$ and $\lambda_1$  in $(K4)$ of Theorem \ref{84thm} may satisfy $\lambda_1 < \lambda$, we are able to prove our Theorems  \ref{84thm} and \ref{Lpthm} with the same $\lambda$.

However, for Theorem \ref{PSL}, we have not yet found a way to prove \eqref{LrLq} by a simplified approach based on Gr{\"o}nwall's inequality, as discussed above. The aforementioned estimate for the term $Gu$, which are based on Pierfelice's dispersive and smooth estimates for the heat kernel $e^{tL}$, are not as useful in proving Theorem \ref{PSL}, and different ideas such as the Marcinkiewicz interpolation theorem and the Hardy-Littlewood-Sobolev lemma are required. 

This leads to at least two disadvantages: the first is that we have found no way to avoid returning to a Picard iteration-type argument to prove \eqref{LrLq} from Theorem \ref{PSL}. The second is that, at least so far, we have not been able to get an exponential decay term in \eqref{LrLq}, leaving this as the sole place where the Euclidean and hyperbolic space theories align exactly. Also, as far as we can tell, to prove Theorem \ref{tmdcy2}, it is necessary to use estimates for the term $Gu$ in the spirit of Kato that involve norms and decay rates of both $u$ and $\nabla u$. 

We also mention here some differences in the semigroup theory between the cases for $\R^n$ and $\mathbb{H}^n$. On Euclidean space $\R^n$, it is well-known that the Laplacian $\Delta_{\R^n}$ generates a strongly continuous semigroup $e^{t\Delta_{\R^n}}$ (see, for example, Chapter 9 of \cite{kato1995}). However, for the semigroup $e^{tL}$ studied in this current work, the strong continuity is not automatic, and, as far as we are aware, has not been shown elsewhere for the case $\mathbb{H}^n$. Therefore, we prove this fact in the Appendix.

Additionally, in the more general setting of $\mathbb{H}^n$, necessary facts such as the $L^p$ boundedness of the projection $\mathbb{P}$ defined by \eqref{projdef} and the commutation of $\mathbb{P}$ and the semigroup $e^{tL}$ are not obvious and without the availability of Fourier transform methods, we must instead resort to spectral theory to prove these results.


\subsection{Organization of the article} \label{orgart}
The rest of the article is structured in the following way. In section \ref{estimatesused} we establish notation, present modified versions of Pierfelice's dispersive and smooth estimates for the semigroup $e^{tL}$, state and prove functional theoretic properties concerning the operators $\mathbb{P}$ and $L$, and prove necessary estimates on the term $Gu$ in \eqref{inteq}.

With these estimates, we then prove Theorem \ref{EUthm} in Section \ref{mildsol} using the method of Picard iteration. Specifically, we use Picard iteration to show \eqref{mainspace} for the case $q > n$, and once this is established, we prove \eqref{mainspace} for the case $q=n$ and the derivative result \eqref{mainDspace} using Gr{\"o}nwall's inequality, as discussed above.

To prove Theorem \ref{PSL}, we first follow the ideas of Giga in \cite{giga_semilinear}, which involves the Marcinkiewicz interpolation theorem to prove \eqref{LrLq} for the free solution $u_0$ in \eqref{inteq}. Then, we return to the sequence of Picard iterates and use induction and the Hardy-Littlewood-Sobolev lemma to extend the result to the solution $u$.

Theorems \ref{newLrLq} and \ref{LrLn} are handled next, and this involves straightforward estimates using the exponential time decay from \eqref{mainspace} in Theorem \ref{EUthm}. Following this, in Section \ref{tmdcy1pf}, we prove Theorem \ref{tmdcy1}, which in the setting of $\mathbb{H}^n$, no longer requires Theorem \ref{PSL}, and is instead an almost immediate consequence of the exponential time decay appearing in Theorem \ref{EUthm}. Then, as previously discussed, we apply the result \eqref{mainspace} for $q > n$ together with Gr{\"o}nwall's inequality in Section \ref{Lpthmpf} to prove Theorem \ref{Lpthm}.

Next, we prove Theorem \ref{LrLp} using the exponential time decay offered by Theorem \ref{Lpthm}, and then we use Theorems \ref{EUthm} and \ref{Lpthm} to prove Theorem \ref{tmdcy2}. Finally, in Section \ref{Lptimedecaypf} we use \ref{Lpthm} to prove Theorem \ref{Lptimedecay}, which in the setting of $\mathbb{H}^n$, is immediate due to the exponential decay term present in \eqref{lpspace} from Theorem \ref{Lpthm}, and, in contrast to the case for $\R^n$, no longer involves an argument necessitating Theorem \ref{tmdcy2}.

\begin{acknowledgment}
The author would like to sincerely thank his thesis advisor Magdalena Czubak for suggesting the problem, for her patience, for her encouragement, and for the many helpful conversations throughout the course of the research.
\end{acknowledgment}

\section{{Notation and estimates used}} \label{estimatesused} 
In this section, we first establish notation and then state the aforementioned dispersive and smoothing estimates of Pierfelice from \cite{Pierfelice} for the semigroup $e^{tL}$. Following this, we must discuss and prove some functional analytic properties of the operator $\mathbb{P}$ defined by \eqref{projdef}, specifically that it is a bounded operator from $L^p$ to $L^p$ and that it commutes with the semigroup $e^{tL}$. 

Both of these facts are discussed in \cite{Pierfelice}, though in the more general setting of that work, Pierfelice must rely on the general Riesz transform boundedness results derived by Lohou{\'e} in \cite{lohoue}. For our more specific setting of $\mathbb{H}^n$, which is a rank-one symmetric space, things are simpler and we can instead rely on the boundedness results for Riesz transforms shown by Strichartz in \cite{Strichartz_laplacian}. 

With these dispersive and smoothing estimates stated and the required functional analytic properties of $\mathbb{P}$ established, we then prove estimates on the term $Gu$ in \eqref{inteq} and its derivatives, which will be essential to our proofs of Theorems \ref{EUthm} - \ref{Lptimedecay}.

To begin, we define some constants that appear in the subsequent dispersive and smoothing estimates. Here and in the rest of the paper, a constant depending on the fixed parameters $a_1, a_2, \ldots, a_k$ is denoted by $C(a_1, a_2, \ldots, a_k)$. For fixed $n, p, q >0$, define

\begin{equation} \label{const}
\begin{split}
\gamma(n,p,q) &= \frac{\delta_n}{2}\bigg[\bigg(\frac{1}{p} - \frac{1}{q}\bigg)+\frac{8}{q}\bigg(1-\frac{1}{p}\bigg) \bigg], \\
\beta_1(n, p, q) &= \frac{\gamma(n,p,q) + c_0}{2}, \\
\beta_2(n, p) & =\frac{4\delta_n}{2p}\bigg(1-\frac{1}{p}\bigg) + \frac{c_0}{2}, \\
\beta_3(n,p,q) &= \frac{1}{4}\big[\gamma(n,q,q) + \gamma(n,p,q)\big] +\frac{c_0}{2},
\end{split}
\end{equation}
where $c_0$ is a positive constant bounding the $\text{Ric}$ operator and the constant $\delta_n > 0$ depends only on the dimension $n$ (for more information on $c_0$ and $\delta_n$, see \cite{Pierfelice} and \cite{setti}). These are various constants used by Pierfelice in \cite{Pierfelice}, only here divided by 2, which simplifies subsequent computations.

As well, going forward, we will use the notation \[L^p(\mathbb{H}^n) \quad \text{(or sometimes just $L^p$)}\]  to denote $L^p$ spaces for functions, forms, and tensors alike and if the distinction is important or necessary at any point, it will be noted.

With this notation established, we can now present the following dispersive and smoothing estimates shown by Pierfelice in \cite{Pierfelice} for $u_0(t) = e^{tL}a$ defined by \eqref{inteq}, and for the specific case $M= \mathbb{H}^n$.

\begin{thm}\cite{Pierfelice}\label{dsmest1}
For all times $t >0$ and $a \in L^p(\mathbb{H}^n)$, $u_0(t) = e^{tL}a$ satisfies the following estimates,

\begin{align} 
\label{dispest1} \big\|u_0(t)\big\|_{L^q(\mathbb{H}^n)} &\leq c_n(t)^{\big(\frac{1}{p}- \frac{1}{q}\big)}e^{-2t\beta_1(n,p,q)}\|a\|_{L^p(\mathbb{H}^n)}, \quad && 1 \leq p \leq q \leq \infty; \\ 
\label{smest11} \big\|\nabla u_0(t)\big\|_{L^p(\mathbb{H}^n)} &\leq C\,\text{\emph{max}}\big(t^{-1/2}, 1\big)e^{-2t\beta_2(n,p)}\|a\|_{L^p(\mathbb{H}^n)}, \quad && 1 < p < \infty; \\
\label{smest21} \big\|\nabla u_0(t)\big\|_{L^q(\mathbb{H}^n)} &\leq c_n(t)^{\big(\frac{1}{p}- \frac{1}{q}+\frac{1}{n}\big)}e^{-2t\beta_3(n,p,q)}\|a\|_{L^p(\mathbb{H}^n)}, \quad && 1 < p \leq q < \infty;
\end{align}
where $c_n(t) = C(n)\text{max }(t^{-n/2},1)$. Moreover, for all tensors $T_0 \in L^p(T\mathbb{H}^n \otimes T^*\mathbb{H}^n)$, we have the following general smoothing estimate
\begin{align} \label{smest31}
\big\|e^{tL}\nabla^*T_0\big\|_{L^q(\mathbb{H}^n)} &\leq c_n(t)^{\big(\frac{1}{p}- \frac{1}{q}+\frac{1}{n}\big)}e^{-2t\beta_3(n,p,q)}\|T_0\|_{L^p(\mathbb{H}^n)}, \quad && 1 < p \leq q < \infty.
\end{align} 
\end{thm}

To prove Theorems \ref{EUthm} - \ref{Lptimedecay}, we will use a modified form of Theorem \eqref{dsmest1}. Specifically, by definition of Pierfelice's constant $c_n(t)$, estimates using Theorem \ref{dsmest1}, especially global in time estimates, will require considering separate cases where $0< t < 1$ and $t \geq 1$. To avoid the technical difficulties presented by this, we instead use the following result, where, for the time regime $t >1$, we trade a factor of $e^{-t\beta_i}$, $i=1,2,3$ for an appropriate inverse power of $t$.

\begin{thm}\label{dsmest}
For all times $t >0$ and $a \in L^p(\mathbb{H}^n)$, $u_0(t) = e^{tL}a$ satisfies the following estimates,

\begin{align} 
\label{dispest} \big\|u_0(t)\big\|_{L^q(\mathbb{H}^n)} &\leq C(n,p,q)\,t^{-\frac{n}{2}\big(\frac{1}{p}- \frac{1}{q}\big)}e^{-t\beta_1(n,p,q)}\|a\|_{L^p(\mathbb{H}^n)}, \quad && 1 \leq p \leq q \leq \infty; \\ 
\label{smest1} \big\|\nabla u_0(t)\big\|_{L^p(\mathbb{H}^n)} &\leq C(n,p)\,t^{-1/2}e^{-t\beta_2(n,p)}\|a\|_{L^p(\mathbb{H}^n)}, \quad && 1 < p < \infty; \\
\label{smest2} \big\|\nabla u_0(t)\big\|_{L^q(\mathbb{H}^n)} &\leq C(n,p,q)\,t^{-\frac{n}{2}\big(\frac{1}{p}- \frac{1}{q}+\frac{1}{n}\big)}e^{-t\beta_3(n,p,q)}\|a\|_{L^p(\mathbb{H}^n)}, \quad && 1 < p \leq q < \infty.
\end{align}
Moreover, for all tensors $T_0 \in L^p(T\mathbb{H}^n \otimes T^*\mathbb{H}^n)$, we have the following general smoothing estimate
\begin{align} \label{smest3}
\big\|e^{tL}\nabla^*T_0\big\|_{L^q(\mathbb{H}^n)} &\leq C(n,p,q)\,t^{-\frac{n}{2}\big(\frac{1}{p}- \frac{1}{q}+\frac{1}{n}\big)}e^{-t\beta_3(n,p,q)}\|T_0\|_{L^p(\mathbb{H}^n)}, \quad && 1 < p \leq q < \infty.
\end{align} 
\end{thm}

\begin{proof}
To prove this modification of Theorem \ref{dsmest1}, we first show that for $\sigma, \beta >0$, there exists a uniform constant $C(\sigma, \beta)$ such that
\begin{equation} \label{ppp}
\frac{t^{\sigma}}{e^{\beta t}} \leq C(\sigma, \beta) \quad \text{ for all } t > 0.
\end{equation}
Letting $f(t) = \frac{t^{\sigma}}{e^{\beta t}}$, elementary calculus shows that $f$ achieves a global maximum on the interval $(0, \infty)$ at $t= \sigma / \beta$ with value 
\begin{equation}
C(\sigma, \beta) := f(\sigma / \beta) = \bigg(\frac{\sigma}{\beta} \bigg)^{\sigma} e^{-\sigma}.
\end{equation}

Therefore, if $t \geq 1$, we set $\sigma = \frac{n}{2}\big(\frac{1}{p}- \frac{1}{q}\big)$ and $\beta =\beta_1(n,p,q)$ in \eqref{ppp} to get
\begin{equation} \label{jjj}
\begin{split}
c_n(t)^{\big(\frac{1}{p}- \frac{1}{q}\big)}e^{-2t\beta_1(n,p,q)} &= C(n)e^{-t\beta_1(n,p,q)}e^{-t\beta_1(n,p,q)} \\
&\leq C(n, p, q)t^{-\frac{n}{2}\big(\frac{1}{p}- \frac{1}{q}\big)}e^{-t\beta_1(n,p,q)}.
\end{split}
\end{equation}

As well, if $0 < t < 1$,
\[
t^{-n/2} > 1, 
\]
so that
\begin{equation} \label{lll}
\begin{split}
c_n(t)^{\big(\frac{1}{p}- \frac{1}{q}\big)}e^{-2t\beta_1(n,p,q)} &= C(n)t^{-\frac{n}{2}\big(\frac{1}{p}- \frac{1}{q}\big)}e^{-2t\beta_1(n,p,q)} \\
&\leq C(n)t^{-\frac{n}{2}\big(\frac{1}{p}- \frac{1}{q}\big)}e^{-t\beta_1(n,p,q)}.
\end{split}
\end{equation}
Combining \eqref{jjj} and \eqref{lll}, and redefining $C(n, \delta, p)$ as the maximum of the two constants appearing in these estimates, implies \eqref{dispest}. 

The estimates \eqref{smest1}, \eqref{smest2}, and \eqref{smest3} are proved analogously. 
\end{proof}

Finally, we must prove estimates on the term $Gu$ appearing in \eqref{inteq} which will be crucial to our subsequent application of Picard iteration. To do this, we must prove some useful facts about the operator $\mathbb{P}$ defined in \eqref{projdef}, the first of which is that it is a bounded operator on $L^p(T^*\mathbb{H}^n)$. To that end, we recall the following theorem of Strichartz from \cite{Strichartz_laplacian}, which concerns the boundedness of Riesz transforms for functions on rank one symmetric spaces.

\begin{thm} \label{symbdd}  \cite{Strichartz_laplacian} 
Let $M$, a complete Riemannian manifold of dimension $n$, be a rank-one symmetric space. Then for any $1 < p < \infty$, $\nabla(-\Delta_g)^{-1/2}$ is a bounded operator from $L^p(M)$, the space of $L^p$ functions on $M$, to $L^p(\mathcal{T}^1_0M)$, the space of $L^p$ tensor fields of rank $(1,0)$.
\end{thm}

Using Theorem \ref{symbdd}, we state and prove  the following boundedness result for $\mathbb{P}$ on $\mathbb{H}^n$ as a corollary.

\begin{cor} \label{Pbdd}
The operator $\mathbb{P} = I - d(-\Delta_g)^{-1}d^*$ is a bounded operator on $L^p(\Omega^1 (\mathbb{H}^n))$, $1<p <\infty$.
\end{cor}

\begin{proof}
Fix $1 < p < \infty$. Since $\mathbb{H}^n$ is a symmetric space of rank one (see for example \cite{geogroup}), Theorem \ref{symbdd} applies for this manifold. Moreover, since $\mathbb{P}$ is defined as the identity minus the differential operator $d(-\Delta_g)^{-1}d^*$, it suffices to prove the $L^p$ boundedness of this latter term.

First we observe that the adjoint of $d(-\Delta_g)^{-1/2}$ is
\[
\big(d(-\Delta_g)^{-1/2}\big)^* = \big((-\Delta_g)^{-1/2}\big)^*d^* = (-\Delta_g)^{-1/2}d^*,
\]
where we have used that $(-\Delta_g)$ is positive and self-adjoint by \cite{Strichartz_laplacian}, which in turn implies its square root $(-\Delta_g)^{-1/2}$ is self-adjoint by the Spectral Theorem (see \cite{conway}, \cite{reedsimon1}, or  \cite{rudinfn}, for example). 

As we have previously shown, for a function $f$, the total covariant derivative $\nabla f$ is given by
\[
\nabla f = df,
\]
so that 
\begin{equation}
\begin{split}
\|\nabla f\|^p_{L^p(\mathbb{H}^n)} &= \int_{\mathbb{H}^n} g(\nabla f, \nabla f)^{p/2}\, dV \\
&= \int_{\mathbb{H}^n} g(d f, d f)^{p/2}\, dV \\
&=  \|d f\|^p_{L^p(\mathbb{H}^n)}.
\end{split}
\end{equation}
Hence Theorem \ref{symbdd} also shows $d(-\Delta_g)^{-1/2}$ is a bounded operator from $L^p$ functions to $L^p$ differential 1-forms on $\mathbb{H}^n$. Explicitly, we have the following
\begin{equation}
\big\|d(-\Delta_g)^{-1/2}\big\|_{op} \leq C(p),
\end{equation}
where $\|\cdot\|_{op}$ denotes the operator norm.

As well, letting $p'$ denote the H{\"o}lder conjugate of $p$, we also have by Theorem \ref{symbdd} that $d(-\Delta_g)^{-1/2}$ is a bounded operator from $L^{p'}$ functions to $L^{p'}$ differential 1-forms. And since $\|d(-\Delta_g)^{-1/2}\|_{op} = \|\big(d(-\Delta_g)^{-1/2}\big)^*\|_{op}$, the dual operator
\begin{equation} \label{rieszbdd}
(-\Delta_g)^{-1/2}d^* = \big(d(-\Delta_g)^{-1/2}\big)^*
\end{equation}
is thus a bounded operator from $L^{p}$ differential 1-forms (the dual space of $L^{p'}$ differential 1-forms) to $L^p$ functions (the dual space of $L^{p'}$ functions), which therefore gives 
\begin{equation} \label{adjbdd}
\big\|(-\Delta_g)^{-1/2}d^*\big\|_{op} \leq C(p').
\end{equation}
Finally, since the Spectral Theorem and functional calculus allow us to write \[(-\Delta_g)^{-1} = (-\Delta_g)^{-1/2}(-\Delta_g)^{-1/2},\] by \eqref{rieszbdd} and \eqref{adjbdd}, we have that for $u \in L^p(\Omega^1 (\mathbb{H}^n))$,
\begin{equation}
\begin{split}
\big\|d(-\Delta_g)^{-1}d^*u\big\|_{L^p(T^*\mathbb{H}^n)} &\leq C(p)C(p') \|u\|_{L^p(\Omega^1 (\mathbb{H}^n))}.
\end{split}
\end{equation}
We conclude that $d(-\Delta_g)^{-1}d^*$ is a bounded operator on $L^p(\Omega^1 (\mathbb{H}^n))$, and thus so is $\mathbb{P}$.
\end{proof}

Having shown $\mathbb{P}$ is bounded on $L^p(\Omega^1 (\mathbb{H}^n))$ for $1 < p < \infty$, we can now show $\mathbb{P}$ commutes with the semigroup $e^{tL}$.

\begin{lem} \label{commute}
On $\mathbb{H}^n$, $e^{tL}\mathbb{P} = \mathbb{P}e^{tL}$ for any $t \in \R$.
\end{lem} 
\begin{proof}
For a fixed $t \in \R$, the function $e^{t\cdot}: \R \to \R$ is Borel measurable on $\R$ and since we have shown in Corollary \ref{Pbdd} that $\mathbb{P}$ is bounded on the Hilbert space $L^2(T^*\mathbb{H}^n)$, it suffices by the Spectral Theorem to show
\begin{equation}
L\mathbb{P} = \mathbb{P}L,
\end{equation}
where the domains $D(\mathbb{P}), D(L),  D(L\mathbb{P})$, and $D(\mathbb{P}L)$ are all taken to be $C_c^{\infty}(T^*\mathbb{H}^n)$ (see, for instance, Theorem 4.11 in Chapter X of \cite{conway}). 

Since we are not assuming any divergence free condition here, we must use the full definition of $L$ given in \eqref{hodgeL}, which for $\mathbb{H}^n$ can be written as
\begin{equation} \label{hodgeL'}
L = - \Delta_{H,1} - dd^* - 2(n-1).
\end{equation}
Therefore, letting $u \in C_c^{\infty}(T^*\mathbb{H}^n)$, and using that $d\Delta_{H,k} = \Delta_{H,k+1}d$ and $d^*\Delta_{H,k} = \Delta_{H,k-1}d^*$, we have
\begin{equation}
\begin{split}
L\mathbb{P}u &= L(u-d(-\Delta_g)^{-1}d^*u) \\
&= Lu+ \Delta_{H,1}d(-\Delta_g)^{-1}d^*u + dd^*d(-\Delta_g)^{-1}d^*u + 2(n-1)d(-\Delta_g)^{-1}d^*u \\
&= Lu+ d\Delta_{H,0}(-\Delta_g)^{-1}d^*u + d(-\Delta_g)(-\Delta_g)^{-1}d^*u + d(-\Delta_g)^{-1}d^*(2(n-1)u) \\
&= Lu+ d(-\Delta_g)(-\Delta_g)^{-1}d^*u + d(-\Delta_g)^{-1}(-\Delta_g)d^*u + d(-\Delta_g)^{-1}d^*(2(n-1)u) \\
&= Lu+ d(-\Delta_g)^{-1}(-\Delta_g)d^*u + d(-\Delta_g)^{-1}d^*dd^*u + d(-\Delta_g)^{-1}d^*(2(n-1)u) \\
&= Lu+ d(-\Delta_g)^{-1}\Delta_{H,0}d^*u + d(-\Delta_g)^{-1}d^*dd^*u + d(-\Delta_g)^{-1}d^*(2(n-1)u) \\
&= Lu+ d(-\Delta_g)^{-1}d^*\Delta_{H,1}u + d(-\Delta_g)^{-1}d^*dd^*u + d(-\Delta_g)^{-1}d^*(2(n-1)u) \\
&= Lu -d(-\Delta_g)^{-1}d^*\big(-\Delta_{H,1}u-dd^*u-2(n-1)u\big) \\
&= Lu -d(-\Delta_g)^{-1}d^*Lu \\
&= \mathbb{P}Lu.
\end{split}
\end{equation}
\end{proof}

With Lemma \ref{commute} in hand, we next show the following useful and necessary result for the divergence of a solution $u$ defined by the integral equation \eqref{INT}.
\begin{prop} \label{divresult}
If $u$ is defined by \eqref{inteq}, then $d^*u = 0$.
\end{prop}
\begin{proof}
It follows by the definition of the operator $\mathbb{P}$ that for a 1-form $w$,
\begin{equation} \label{Pdivfree}
\begin{split}
d^*(\mathbb{P}w) &= d^*w - d^*d(-\Delta_g)^{-1}d^*w \\
&= d^*w - (-\Delta_g)(-\Delta_g)^{-1}d^*w \\
&= d^*w - d^*w \\
&=0.
\end{split}
\end{equation}
As well, for our initial condition $a \in L^n(\mathbb{H}^n)$,
\begin{equation}
\mathbb{P}a = a - d(-\Delta_g)^{-1}d^*a = a,
\end{equation}
since $d^*a = 0$. Thus by Lemma \ref{commute}
\begin{equation}
\begin{split}
d^*u(t) &= d^*e^{tL}a - d^*\int_0^te^{(t-s)L}\mathbb{P }\,\big(\text{div}(u \otimes u)\big)(s)\,ds \\
&= d^*e^{tL}(\mathbb{P}a) - \int_0^td^*e^{(t-s)L}\mathbb{P }\,\big(\text{div}(u \otimes u)\big)(s)\,ds \\
&= d^*\mathbb{P}e^{tL}a - \int_0^td^*\mathbb{P}e^{(t-s)L}\,\big(\text{div}(u \otimes u)\big)(s)\,ds \\
&=0,
\end{split}
\end{equation}
where we have used \eqref{Pdivfree} in the last step.
\end{proof}

With these facts about the operator $\mathbb{P}$ established, we can now state and prove the following important estimates on the term $Gu(t)$, which will be used often in the sequel.

\begin{lem} \label{gtermest}
Let $u$ be defined by equation \eqref{INT}. Then $Gu$ and $\nabla Gu$ satisfy the following estimates for  $0<\gamma \leq \alpha + \zeta < n$, where  $\beta_4 := \beta_3\big(n, \frac{n}{\alpha+\zeta},\frac{n}{\gamma}\big)$:
\begin{enumerate}
\item 
\begin{equation} \label{Gest}
\|Gu(t)\|_{L^{n/\gamma}} \leq C(n, \alpha, \gamma, \zeta)\int_0^t (t-s)^{-\frac{\alpha +\zeta - \gamma +1}{2}}e^{-(t-s)\beta_4}\|u(s)\|_{L^{n/\alpha}}\|u(s)\|_{L^{n/\zeta}}\,ds,
\end{equation}
\item 
\begin{equation} \label{Gdest}
\|Gu(t) - Gv(t)\|_{L^{n/\gamma}} \leq C(n, \alpha, \gamma, \zeta) \int_0^t (t-s)^{-\frac{\alpha +\zeta - \gamma +1}{2}}e^{-(t-s)\beta_4} \big(\|u\|_{L^{n/\alpha}}+\|v\|_{L^{n/\alpha}}\big) \|u-v\|_{L^{n/\zeta}}\,ds,
\end{equation}
\item 
\begin{equation} \label{DGest}
\|\nabla Gu(t)\|_{L^{n/\gamma}} \leq C(n, \alpha, \gamma, \zeta)\int_0^t (t-s)^{-\frac{\alpha +\zeta - \gamma +1}{2}}e^{-(t-s)\beta_4}\|u(s)\|_{L^{n/\alpha}}\|\nabla u(s)\|_{L^{n/\zeta}}\,ds,
\end{equation}
\item 
\begin{multline} \label{DGdest}
\|\nabla Gu(t) - \nabla Gv(t)\|_{L^{n/\gamma}} \leq C(n, \alpha, \gamma, \zeta) \int_0^t (t-s)^{-\frac{\alpha +\zeta - \gamma +1}{2}}e^{-(t-s)\beta_4} \big(\| u\|_{L^{n/\alpha}}\|\nabla u - \nabla v\|_{L^{n/\zeta}} \\+ \| \nabla v\|_{L^{n/\zeta}}\|u- v\|_{L^{n/\alpha}}\big)\,ds.
\end{multline}
\end{enumerate}
\end{lem}

\begin{proof}
We first observe that by Proposition \ref{divresult}, if $u$ is defined by \eqref{INT}, then $d^*u = 0$, so that we may move freely back and forth between
\[
\nabla_{u^{\#}} u \quad \text{ and } \quad (\text{div} (u^{\#} \otimes u^{\#}))^{\flat}
\] 
as needed, since these expressions are equal whenever $d^*u = 0$. Next we observe that for $T \in T\mathbb{H}^n \otimes T^*\mathbb{H}^n$,
\begin{equation}
\nabla^*T = -\text{div}(T^{\#}),
\end{equation}
so that \eqref{smest3} implies 
\begin{equation} \label{smest6}
\big\|e^{tL}\text{div}(T^{\#})\big\|_{L^q(\mathbb{H}^n)} \leq C(n,p,q)\,t^{-\frac{n}{2}\big(\frac{1}{p}- \frac{1}{q}+\frac{1}{n}\big)}e^{-t\beta_3(n,p,q)}\|T^{\#}\|_{L^p(\mathbb{H}^n)} \quad 1 < p \leq q < \infty,
\end{equation}
for all tensors $T^{\#} \in TM \otimes TM$, where now $\#$ represents the induced musical isomorphism $(\cdot)^\#: T\mathbb{H}^n \otimes T^*\mathbb{H}^n \to T\mathbb{H}^n \otimes T\mathbb{H}^n$ (see, for example, \cite{Pierfelice}). 

By commuting $\mathbb{P}$ and $e^{(t-s)L}$ according to Lemma \ref{commute}, using the $L^{q}$ boundedness of $\mathbb{P}$ from Corollary \ref{Pbdd}, and applying \eqref{smest6} along with H{\"o}lder's inequality, we have that
\begin{equation}
\begin{split}
\|Gu(t)\|_{L^q(\mathbb{H}^n)} &\leq C(n,p,q)\int_0^t (t-s)^{-\frac{n}{2}\big(\frac{1}{p}- \frac{1}{q}+\frac{1}{n}\big)}e^{-(t-s)\beta_3(n,p,q)}\|u^{\#}\otimes u^{\#}(s)\|_{L^p}\,ds \\
&\leq C(n,\mu, \lambda, q)\int_0^t (t-s)^{-\frac{n}{2}\big(\frac{1}{\mu}+\frac{1}{\lambda}- \frac{1}{q}+\frac{1}{n}\big)}e^{-(t-s)\beta_3\big(n,\frac{\mu\lambda}{\mu+\lambda},q\big)}\|u(s)\|_{L^{\mu}}\|u(s)\|_{L^{\lambda}}\,ds,
\end{split}
\end{equation}
where $\lambda, \mu>0$ and $\frac{1}{\mu} +\frac{1}{\lambda} = \frac{1}{p}$ and where we have also used that $(\cdot)^{\#}$ and $(\cdot)^{\flat}$ are isomorphisms.

Taking this inequality and setting
\begin{equation}
q= \frac{n}{\gamma}, \quad \mu = \frac{n}{\alpha}, \quad \lambda = \frac{n}{\zeta},
\end{equation}
where $\gamma \leq \alpha +\zeta <n$, proves \eqref{Gest}. As well, using the fact that
\[
u^{\#}\otimes u^{\#} - v^{\#}\otimes v^{\#} = u^{\#}\otimes (u-v)^{\#} - (v-u)^{\#}\otimes v^{\#},
\]
a slight modification of the argument leading to \eqref{Gest} gives the difference estimate \eqref{Gdest}.

The derivative estimate \eqref{DGest} is simpler, for in this case we only need the $L^q$ boundedness of $\mathbb{P}$ from Corollary \ref{Pbdd} and the following pointwise bound for $C^1$ 1-forms $v_1, v_2$: 
\begin{equation}  \label{covarest}
g(\nabla_{v_2^{\#}}v_1 , \nabla_{v_2^{\#}} v_1) \leq |\nabla v_1|^2 |v_2|^2,
\end{equation}
which can be shown much like it would in the Euclidean case by using geodesic normal coordinates (here, $| \cdot | = \sqrt{g(\cdot, \cdot)}$).

Indeed, using Corollary \ref{Pbdd} and the pointwise estimate \eqref{covarest}, passing the covariant derivative through the time integral defining $Gu$, applying estimate \eqref{smest2} from Theorem \ref{dsmest}, and using H{\"o}lder's inequality, we have
\begin{equation}
\begin{split}
\|\nabla Gu(t) \|_{L^q(\mathbb{H}^n)}  &\leq \int_0^t \big\| \nabla e^{(t-s)L} (\mathbb{P} \nabla_{u^{\#}} u(s)) \|_{L^q} \,ds \\
& \leq C(n, p, q) \int_0^t (t-s)^{-\frac{n}{2}\big(\frac{1}{p}- \frac{1}{q}+\frac{1}{n}\big)}e^{-(t-s)\beta_3(n,p,q)}\big\| |\nabla u | |u|\big\|_{L^p}  \\
&\leq C(n,\mu, \lambda, q)\int_0^t (t-s)^{-\frac{n}{2}\big(\frac{1}{\mu}+\frac{1}{\lambda}- \frac{1}{q}+\frac{1}{n}\big)}e^{-(t-s)\beta_3\big(n,\frac{\mu\lambda}{\mu+\lambda},q\big)}\|u(s)\|_{L^{\mu}}\|\nabla u(s)\|_{L^{\lambda}}\,ds,
\end{split}
\end{equation}
where as before, $\lambda, \mu>0$ and $\frac{1}{\mu} +\frac{1}{\lambda} = \frac{1}{p}$. We finish the proof of \eqref{DGest} by once again setting 
\begin{equation}
q= \frac{n}{\gamma}, \quad \mu = \frac{n}{\alpha}, \quad \lambda = \frac{n}{\zeta}.
\end{equation}
It remains to show \eqref{DGdest} and for this, we can proceed much as in the proof of \eqref{DGest}, again using the pointwise estimate \eqref{covarest} and the fact that 
\[
\nabla_{u^{\#}} u - \nabla_{v^{\#}} v = \nabla_{u^{\#}}(u-v) - \nabla_{(v-u)^{\#}}v.
\]
\end{proof}

\section{Picard iteration on $\mathbb{H}^n$ and proof of Theorem \ref{EUthm}}\label{mildsol} 
\subsection{Local existence for $q >n$} \label{mainexist}
We will solve \eqref{INT} and thus prove Theorem \ref{EUthm} by Picard iteration. Starting with $u_0(t) = e^{tL}a$, where $a \in L^n(\mathbb{H}^n)$ is the initial condition, we construct the following sequence
 \begin{equation} \label{iterate}
 u_{k+1} = u_0+Gu_k, \quad k=0,1,2,3,\ldots
 \end{equation}
Let $0 < \delta < 1$ be fixed. We will first show by induction that the sequence defined by \eqref{iterate} exists and satisfies
\begin{equation} \label{deltaest}
t^{\big(\frac{1}{2}- \frac{\delta}{2} \big)}e^{t\beta} u_k \in L^{\infty}\big([0,T], L^{n/\delta}(\mathbb{H}^n)\big),
\end{equation}  
with norm
\begin{equation} \label{Lnseqnorm}
M_k := \sup_{0 \leq t < T} t^{\big(\frac{1}{2}-\frac{\delta}{2}\big)}e^{t\beta}\|u_k(t)\|_{L^{n/\delta}(\mathbb{H}^n)}, \quad k=0,1,2,3,\ldots,
\end{equation}
where $T>0$ is to be chosen, and 
\begin{equation} \label{1stbeta}
\beta = \min\bigg\{ \beta_1\bigg(n,n,\frac{n}{\delta}\bigg),  \beta_3\bigg(n,\frac{n}{2\delta},\frac{n}{\delta}\bigg)\bigg\}. 
\end{equation}
Notice with $\beta$ so defined, we have $\beta >0$, which can easily be verified by the definitions given in \eqref{const}. As well, for $t \geq 0$, we have
\begin{equation}
e^{-t \beta_1\big(n,n,\frac{n}{\delta}\big)} \leq e^{-t\beta} \quad \text{and} \quad e^{-t \beta_3\big(n,\frac{n}{2\delta},\frac{n}{\delta}\big)} \leq e^{-t\beta},
\end{equation}
which are estimates that will be used liberally in subsequent computations.

For $k=0$, we set $p=n$ and $q=n/\delta$ in \eqref{dispest} to get
\begin{equation} \label{u0iter}
\begin{split}
\big\|u_0(t)\big\|_{L^{n/\delta}(\mathbb{H}^n)} &\leq C(n,\delta)\,t^{-\big(\frac{1}{2}- \frac{\delta}{2}\big)}e^{-t\beta_1\big(n,n,\frac{n}{\delta}\big)}\|a\|_{L^n(\mathbb{H}^n)} \\
&\leq C(n,\delta)\,t^{-\big(\frac{1}{2}- \frac{\delta}{2}\big)}e^{-t\beta}\|a\|_{L^n(\mathbb{H}^n)}. \\
\end{split} 
\end{equation}
Thus 
\begin{equation} \label{M0bound}
M_0  \leq C(n,\delta)\|a\|_{L^n(\mathbb{H}^n)},
\end{equation}
and \eqref{deltaest} is satisfied for $k=0$. 

Assuming now that \eqref{deltaest} is true for $k>0$, we next show it holds for $k+1$. Since the first term in \eqref{iterate} has just been estimated, we must estimate the term $Gu_k(t)$ in the $L^{n/\delta}(\mathbb{H}^n)$ norm. To do this, we set $\alpha = \gamma = \zeta = \delta$ in \eqref{Gest}, and use that $e^{-\beta t} \leq 1$ for all $t \geq 0$, which gives
\begin{equation} \label{ndeltacomp1}
\begin{split}
\|Gu_k(t)\|_{L^{n/\delta}(\mathbb{H}^n)} &\leq C(n,\delta) \int_0^t (t-s)^{-(1+\delta)/2} e^{-(t-s)\beta_3\big(n,\frac{n}{2\delta},\frac{n}{\delta}\big)}\big(\|u_k(s)\|_{L^{n/\delta}(\mathbb{H}^n)}\big)^2\,ds \\
&\leq C(n,\delta)M_k^2\int_0^t (t-s)^{-(1+\delta)/2}e^{-(t-s)\beta} \big(s^{-\big(\frac{1}{2}-\frac{\delta}{2}\big)}e^{-\beta s}\big)^2\,ds \\
&= C(n,\delta)M_k^2e^{-t\beta}\int_0^t (t-s)^{-(1+\delta)/2}s^{-1+\delta}e^{-\beta s}  \,ds \\
&\leq C(n,\delta)M_k^2e^{-t\beta}\int_0^t (t-s)^{-(1+\delta)/2}s^{-1+\delta} \,ds.
\end{split}
\end{equation}
To compute the terminal integral in this inequality, we use the Beta function (see, for instance, \cite{math_meth_phys}), which is defined for $x, y \in \mathbb{C}$ by
\begin{equation} \label{betadef}
B(x,y) = \int_0^1 {\tau}^{x-1}(1-\tau)^{y-1}\,d\tau,
\end{equation}
and is finite whenever $\text{Re} (x),\,\text{Re} (y) >0$. After performing the substitution $\tau = s/t$, the terminal integral in \eqref{ndeltacomp1} can be rewritten as
\begin{equation} \label{usub}
\begin{split}
\int_0^t (t-s)^{-(1+\delta)/2}s^{-1+\delta}  \,ds &= t^{-\big(\frac{1}{2}-\frac{\delta}{2}\big)} \int_0^1 \tau^{\delta -1} (1-\tau)^{\big(\frac{1-\delta}{2}\big) -1} \,d\tau \\
&=  t^{-\big(\frac{1}{2}-\frac{\delta}{2}\big)} B\bigg(\delta, \frac{1-\delta}{2}\bigg),
\end{split}
\end{equation}
and $B\big(\delta, \frac{1-\delta}{2}\big)$ is finite, since both $0 < \delta <1$ and $ 0 < \frac{1-\delta}{2} <1$. Hence 
\begin{equation} \label{ndeltacomp}
\begin{split}
\|Gu_k(t)\|_{L^{n/\delta}(\mathbb{H}^n)} &\leq C(n,\delta)M_k^2 t^{-\big(\frac{1}{2}-\frac{\delta}{2}\big)}e^{-t\beta},
\end{split}
\end{equation}
where we have have absorbed the Beta function $B\big(\delta, \frac{1-\delta}{2}\big)$ into the constant $C(n, \delta)$,  Thus we have shown
\begin{equation}
t^{\big(\frac{1}{2}-\frac{\delta}{2}\big)}e^{t\beta}\|u_{n+1}(t)\|_{L^{n/\delta}(\mathbb{H}^n)} \leq M_0 + C(n,\delta)M_k^2,
\end{equation}
which proves \eqref{deltaest} for $k+1$ and leads to the recurrence inequality
\begin{equation} \label{kplus1}
M_{k+1} \leq M_0 + C(n, \delta)M_k^2.
\end{equation}
We claim this defines a bounded sequence $\{M_k\}$ where
\begin{equation} \label{M0global}
M_k  < M:= \frac{1}{2C(n, \delta)}, \quad k=0,1,2,3,\ldots,
\end{equation}
provided
\begin{equation} \label{estrequire}
M_0 < \frac{1}{4C(n,\delta)},
\end{equation}
which is possible according to \eqref{Lnseqnorm} by choosing $T>0$ sufficiently small. Indeed, for the base case, we have
\[
M_0 < \frac{1}{4C(n,\delta)} < \frac{1}{2C(n,\delta)},
\]
so that if \eqref{M0global} is true for $k>0$, then by the recurrence inequality \eqref{kplus1},
\begin{equation}
\begin{split}
M_{k+1} &\leq M_0 + C(n, \delta)M_k^2 \\
& <\frac{1}{4C(n,\delta)} + C(n, \delta)\bigg(\frac{1}{2C(n,\delta)}\bigg)^2 \\
&= \frac{1}{2C(n,\delta)},
\end{split}
\end{equation}
which verifies \eqref{M0global} for $k+1$ and thus the entire sequence of norms $\{M_k\}$.

With $T>0$ chosen so that \eqref{estrequire} holds, we next show \eqref{M0global} implies the sequence $u_k$ defined by \eqref{iterate} converges uniformly. Define a new sequence
\begin{equation}
w_k(t) = u_k(t) - u_{k-1}(t), \quad k=0,1,2,\ldots,
\end{equation} 
where $u_{-1}(t) =0$. We will prove by induction that for each $k=0,1,2,\ldots$,
\begin{equation} \label{maindest}
\|w_{k}(t)\|_{L^{n/\delta}(\mathbb{H}^n)} \leq \frac{M(2C(n,\delta)M)^k}{t^{(1-\delta)/2}e^{t\beta}}
\end{equation}
for some constant $C(n,\delta)$. For $k=0$, we have by \eqref{M0global} that $M_0 \leq M$. Hence for $0 \leq t \leq T$,
\begin{equation}
\begin{split}
t^{(1-\delta)/2}e^{t\beta}\|w_0(t)\|_{L^{n/\delta}(\mathbb{H}^n)}  &= t^{(1-\delta)/2}e^{t\beta}\|u_0(t)\|_{L^{n/\delta}(\mathbb{H}^n)} \\
&\leq M,
\end{split}
\end{equation}
so that
\begin{equation}
\begin{split}
\|w_0(t)\|_{L^{n/\delta}(\mathbb{H}^n)}  &\leq t^{-(1-\delta)/2}e^{-t\beta }M \\
&\leq \frac{M(2C(n,\delta)M)^0}{t^{(1-\delta)/2}e^{t\beta}}.
\end{split}
\end{equation}

Assuming that \eqref{maindest} holds for $k>0$, we estimate the $(k+1)$-th term. Setting $\alpha = \gamma = \zeta =  \delta$ in \eqref{Gdest} and using \eqref{M0global}, we can proceed as in \eqref{ndeltacomp1} to get
\begin{equation}
\begin{split}
\|w_{k+1}(t)\|_{L^{n/\delta}(\mathbb{H}^n)} &=\|u_{k+1}(t)-u_k(t)\|_{L^{n/\delta}(\mathbb{H}^n)} \\ 
&= \|Gu_{k}(t) - Gu_{k-1}(t)\|_{L^{n/\delta}(\mathbb{H}^n)} \\
& \leq C(n,\delta) \int_0^t (t-s)^{-(1+\delta)/2}e^{-(t-s)\beta } \big(\|u_{k}(s)\|_{L^{n/\delta}}+\|u_{k-1}(s)\|_{L^{n/\delta}}\big) \|w_{k}(s)\|_{L^{n/\delta}}\,ds \\
& \leq 2C(n,\delta)M e^{-t\beta }\int_0^t (t-s)^{-(1+\delta)/2} s^{-(1-\delta)/2} \|w_{k}(s)\|_{L^{n/\delta}(\mathbb{H}^n)}\,ds \\
&\leq (2C(n,\delta)M)(M(2C(n,\delta)M)^k)e^{-t\beta } \int_0^t (t-s)^{-(1+\delta)/2} s^{-(1-\delta)/2} s^{-(1-\delta)/2}e^{-\beta s}\,ds \\
&\leq (2C(n,\delta)M)(M(2C(n,\delta)M)^k)e^{-t\beta } \int_0^t (t-s)^{-(1+\delta)/2} s^{-(1-\delta)/2} s^{-(1-\delta)/2}\,ds \\
&= (2C(n,\delta)M)(M(2C(n,\delta)M)^k) e^{-t\beta }\int_0^t (t-s)^{-(1+\delta)/2} s^{-1+\delta} \,ds \\
&=\frac{M(2C(n,\delta)M)^{k+1}}{t^{(1-\delta)/2}e^{t\beta }},
\end{split}
\end{equation}
where in the last step, we have computed the integral using the same Beta function computation as employed in \eqref{usub} and have absorbed the resulting convergent Beta function into the general constant $C(n, \delta)$.

This verifies \eqref{maindest} for all $k =0, 1 ,2 \ldots$ and since $2C(n,\delta)M < 1$ by assumption, we have that $\displaystyle{\sum_{j=0}^{\infty}w_k}$ converges uniformly and absolutely to 1-form $u$ such that $ t^{\big(\frac{1}{2}- \frac{\delta}{2} \big)}e^{t\beta } u \in L^{\infty}\big([0,T], L^{n/\delta}(\mathbb{H}^n)\big)$ and such that
\begin{equation} \label{uniform}
\begin{split}
u &= \lim_{k \to \infty} \sum_{j=0}^k w_j \\
&= \lim_{k \to \infty} u_k.
\end{split}
\end{equation}

\subsection{Uniqueness}
Having concluded the local existence part of Theorem \ref{PFthm}, we now deal with uniqueness. Let $a, a'\in L^n(\mathbb{H}^n)$ and suppose $u$ and $v$ are the solutions to \eqref{INT} for these initial data, with respective times of existence $T$ and $T'$. Then
\begin{equation}
\begin{split}
u(t) &= u_0(t) + Gu(t), \\
v(t) &= v_0(t) + Gv(t),
\end{split}
\end{equation}
where $u_0(t) = e^{tL}a$ and $v_0(t) = e^{tL}a'$. Letting
\[
w(t) = u(t) - v(t),
\] 
using dispersive estimate \eqref{dispest} and difference estimate \eqref{Gdest} with $\alpha = \gamma = \zeta = \delta$, and performing similar calculations as above, we have
\begin{equation} 
\begin{split}
\|w(t) \|_{L^{n/\delta}(\mathbb{H}^n)} &\leq \|e^{tL}(a-a'))\|_{L^{n/\delta}(\mathbb{H}^n)} + C(n,\delta)Me^{-t\beta}\int_0^t (t-s)^{-\frac{(1+\delta)}{2}}s^{-\frac{(1-\delta)}{2}} \|w(s)\|_{L^{n/\delta}(\mathbb{H}^n)}\,ds \\
&=C(n,\delta) t^{-\frac{(1-\delta)}{2}}e^{-t\beta }\|a-a'\|_{L^{n}(\mathbb{H}^n)}+t^{-\frac{(1-\delta)}{2}}e^{-t\beta}\int_0^t g(t,s) s^{\frac{(1-\delta)}{2}} \|w(s)\|_{L^{n/\delta}(\mathbb{H}^n)}\,ds,
\end{split}
\end{equation} 
where $g(t,s) =C(n,\delta)M(t-s)^{-(\delta+1)/2}s^{-(1-\delta)}t^{(1-\delta)/2}$. Multiplying through by $t^{(1-\delta)/2}e^{t\beta}$ in the above gives
\begin{equation}\label{pregron} 
t^{\frac{(1-\delta)}{2}} e^{t\beta}\|w(t) \|_{L^{n/\delta}(\mathbb{H}^n)} \leq C(n,\delta) \|a-a'\|_{L^{n}(\mathbb{H}^n)}+\int_0^t g(t,s) s^{\frac{(1-\delta)}{2}} \|w(s)\|_{L^{n/\delta}(\mathbb{H}^n)}\,ds,
\end{equation}
and after applying Gr{\"o}nwall's inequality, we get
\begin{equation}
t^{\frac{(1-\delta)}{2}}e^{t\beta}\|w(t) \|_{L^{n/\delta}(\mathbb{H}^n)} \leq C(n,\delta)\|a-a'\|_{L^n(\mathbb{H}^n)}\text{exp}\bigg(\int_0^t g(t,s)\,ds\bigg),
\end{equation}
so that defining $\tilde{T} = \min\{T,T'\}$,
\begin{equation}  \label{gron}
\begin{split}
\sup_{0 \leq t < \tilde{T}} t^{\frac{(1-\delta)}{2}}e^{t\beta}\|w(t) \|_{L^{n/\delta}} &\leq C(n,\delta)\|a-a'\|_{L^n}\text{exp}\bigg(\int_0^t g(t,s)\,ds\bigg) \\
& =C(n,\delta)\|a-a'\|_{L^n}\text{exp}\bigg(C(n,\delta)Mt^{(1-\delta)/2} \int_0^t (t-s)^{-(\delta+1)/2}s^{-(1-\delta)} \,ds\bigg) \\
& =C(n,\delta)\|a-a'\|_{L^n}\text{exp}\bigg(C(n,\delta)Mt^{(1-\delta)/2}t^{-(1-\delta)/2}B\bigg(\delta,\frac{1-\delta}{2}\bigg) \bigg) \\
& =C(n,\delta)\|a-a'\|_{L^n}\text{exp}\bigg(C(n,\delta)MB\bigg(\delta,\frac{1-\delta}{2}\bigg) \bigg) \\
&=C(n, \delta, M) \|a-a'\|_{L^n},
\end{split}
\end{equation} 
where we have used the same Beta function computation as in \eqref{usub}. Thus in the case where $a = a'$, \eqref{gron} shows $w = 0$ in $L^{\infty}([0,\tilde{T}), L^{n/\delta}(\mathbb{H}^n))$, so that $u = v$ and we can take $T=T'=\tilde{T}$. This proves uniqueness.

\subsection{Continuous dependence on initial data}
Let $\varepsilon >0$ and suppose $a, a' \in L^{n}(\mathbb{H}^n)$. If $u$ is the solution constructed in the preceding sections for $a$ with time of existence $T$ and if $v$ is the corresponding solution for $a'$ with time of existence $T'$, then letting $\tilde{T} = \min\{T,T'\}$, our Gr{\"o}nwall estimate \eqref{gron} shows
\begin{equation}
\sup_{0 \leq t < \tilde{T}} t^{\frac{(1-\delta)}{2}}e^{t\beta}\|u(t) - v(t) \||_{L^{n/\delta}(\mathbb{H}^n)}\leq C(n, \delta, M) \|a-a'\|_{L^n(\mathbb{H}^n)}.
\end{equation}
Therefore, if $a$ and $a'$ are such
\begin{equation}
\|a-a'\|_{L^{n}(\mathbb{H}^n)} < \frac{\eps}{C(n, \delta, M)},
\end{equation}
then
\begin{equation}
\sup_{0 \leq t < \tilde{T}} t^{\frac{(1-\delta)}{2}}e^{t\beta}\|u(t) - v(t) \|_{L^{n/\delta}} < \eps,
\end{equation}
which establishes continuous dependence on the initial data.

\subsection{Global well-posedness}
For global existence and uniqueness, we next observe by virtue of \eqref{M0bound}, \eqref{M0global}, and \eqref{estrequire} that if 
\begin{equation} \label{asmall}
\|a\|_{L^n(\mathbb{H}^n)} < \frac{1}{4(C(n, \delta))^2},
\end{equation}
then $M_0 < 1/4C(n,\delta)$ for any choice of $T>0$. In this case, the uniform convergence of $u_k$ to $u$ can be shown on the interval $[0,\infty)$.

\subsection{The case q=n} \label{caseq=n}
\subsubsection{Estimates}
So far, we have shown existence and uniqueness of a solution $u$ of \eqref{INT} satisfying
\begin{equation}
t^{\big(\frac{1}{2}- \frac{n}{2q} \big)}e^{t\beta} u \in L^{\infty}\big([0,T), L^{q}(\mathbb{H}^n)\big), \quad n < q < \infty,
\end{equation}  
for some $\beta >0$ depending only on $n$ and $q$, since, given any $q >n$, $\delta = n/q$ satisifes $0<\delta<1$ as assumed above. As well, we have shown we can extend $T$ to infinity if $\|a\|_{L^n}$ is sufficiently small. Therefore, it remains to show that the limit $u$ satisfies $e^{t\beta'}u \in L^{\infty}\big([0,T), L^{n}(\mathbb{H}^n)\big)$, where for a fixed $0 < \delta <1$, 
\begin{equation} \label{2ndbeta}
\beta' = \min\bigg\{\beta, \beta_1(n,n,n), \beta_3\bigg(n,\frac{n}{\delta+1},n\bigg)\bigg\},
\end{equation}
with $\beta$ being defined by \eqref{1stbeta}.

To that end, by
\eqref{dispest} with $q=p=n$, 
\begin{equation}\label{q=nu0est}
\begin{split}
\|u_0(t)\|_{L^{n}(\mathbb{H}^n)} &\leq C(n)e^{-t\beta_1(n,n,n)}\|a\|_{L^{n}(\mathbb{H}^n)} \\
&\leq C(n)e^{-t\beta'}\|a\|_{L^{n}(\mathbb{H}^n)}.
\end{split}
\end{equation}

Next, if $T$ is chosen so that \eqref{estrequire} holds, then \eqref{deltaest} and \eqref{M0global} imply
\[
\|u(t)\|_{L^{n/\delta}(\mathbb{H}^n)} < Mt^{-\big(\frac{1}{2} - \frac{\delta}{2}\big)}e^{-t\beta'}, \quad \text{for } 0 \leq t < T.
\]
Hence, using \eqref{Gest} with $\alpha  = \delta$ for a fixed $0 < \delta < 1$ and $\gamma = \zeta= 1$, it follows that
\begin{equation} \label{Lnest} 
\begin{split} 
\|Gu(t)\|_{L^n(\mathbb{H}^n)} &\leq C(n,\delta)\int_0^t (t-s)^{-\frac{(\delta+1)}{2}}e^{-(t-s)\beta_3\big(n,\frac{n}{\delta+1},n\big)} \|u(s)\|_{L^{n/\delta}(\mathbb{H}^n)}\|u(s)\|_{L^{n}(\mathbb{H}^n)}\,ds \\
& \leq C(n,\delta)M\int_0^t (t-s)^{-\frac{(\delta+1)}{2}}e^{-(t-s)\beta'}s^{-\frac{(1-\delta)}{2}}e^{-s\beta' } \|u(s)\|_{L^{n}(\mathbb{H}^n)}\,ds \\
&=C(n,\delta)Me^{-t\beta'}\int_0^t (t-s)^{-\frac{(\delta+1)}{2}}s^{-\frac{(1-\delta)}{2}}\|u(s)\|_{L^{n}(\mathbb{H}^n)}\,ds.
\end{split}
\end{equation}
By combining this estimate with \eqref{q=nu0est} and using that $u(t) = u_0(t) + Gu(t)$, we have shown that
\begin{equation}
e^{t\beta'}\|u(t)\|_{L^n(\mathbb{H}^n)} \leq C(n)\|a\|_{L^{n}(\mathbb{H}^n)} + C(n,\delta)M\int_0^t (t-s)^{-\frac{(\delta+1)}{2}}s^{-\frac{(1-\delta)}{2}} \|u(s)\|_{L^{n}(\mathbb{H}^n)}\,ds,
\end{equation}
so that by Gr{\"o}nwall's inequality,
\begin{equation}
\begin{split}
e^{t\beta'}\|u(t)\|_{L^n(\mathbb{H}^n)} &\leq C(n)\|a\|_{L^{n}(\mathbb{H}^n)}\text{exp}\bigg( C(n,\delta)M\int_0^t (t-s)^{-\frac{(\delta+1)}{2}}s^{-\frac{(1-\delta)}{2}} \,ds\bigg) \\
&\leq C(n)\|a\|_{L^{n}(\mathbb{H}^n)}\text{exp}\bigg( C(n,\delta)MB\bigg(\frac{\delta+1}{2}, \frac{1-\delta}{2}\bigg)\bigg) \\
\end{split}
\end{equation}
where we have made the substitution $\tau = s/t$. This shows
\begin{equation} \label{mmm}
\sup_{0 \leq t < T}e^{t\beta'}\|u(t)\|_{L^n(\mathbb{H}^n)} \leq C(n)\|a\|_{L^{n}(\mathbb{H}^n)}\text{exp}\bigg( C(n,\delta)MB\bigg(\frac{\delta+1}{2}, \frac{1-\delta}{2}\bigg)\bigg),
\end{equation}
so that $e^{t\beta'}u \in L^{\infty}([0,T), L^n(\mathbb{H}^n))$. 

For the global estimate,  if $\|a\|_{L^{n}(\mathbb{H}^n)}$ satisfies the smallness condition \eqref{asmall}, then by the arguments outlined in Section \ref{mainexist},
\[
\|u(t)\|_{L^{n/\delta}(\mathbb{H}^n)} < Mt^{-\big(\frac{1}{2} - \frac{\delta}{2}\big)}e^{-t\beta'}, \quad \text{for } 0 \leq t < \infty,
\]
and the same computations just performed for the local estimate show $e^{t\beta'}u \in L^{\infty}([0,\infty), L^n(\mathbb{H}^n))$.

\subsubsection{Uniqueness and continuous dependence on initial data}
For the case $q=n$, the proofs for uniqueness and continuous dependence on initial data are analogous to the proofs for the case $q > n$.



\subsection{Estimates for $\nabla u$} It remains to prove \eqref{mainDspace} holds for $n \leq q < \infty$ and for this, we first choose $0< \delta < 1$ such that
\begin{equation} \label{qbdd2}
n \leq q < \frac{n}{1-\delta}.
\end{equation}
Note that if $q=n$, then any $0 <\delta < 1$ will work, whereas if $q >n$, we can choose $0<\varepsilon < \frac{n}{q}$ and define
\[
\delta = 1+\varepsilon - \frac{n}{q}.
\]
Then $0<\delta < 1$ and
\[
q = \frac{n}{1-\delta+\varepsilon} < \frac{n}{1-\delta}.
\]

With this $\delta$ chosen, we will show the following: 
\begin{equation}
t^{1-\frac{n}{2q}}e^{t\beta''}\nabla u  \in BC\big([0,T), L^q(\mathbb{H}^n)\big),
\end{equation}
where $\beta''>0$ is defined by
\begin{equation} \label{3rdbeta}
\beta'' = \min\bigg\{\beta, \beta_3(n,n,q), \beta_3\bigg(n,\frac{n}{n/q+\delta},q\bigg) \bigg\},
\end{equation}
with $\beta$ given by \eqref{1stbeta}.

To estimate the derivative of the free solution $u_0(t)$, we take $p=n$ in smoothing estimate \eqref{smest2} to get 
\begin{equation} \label{LqD0}
\begin{split}
\|\nabla u_0(t)\|_{L^q(\mathbb{H}^n)} &\leq C(n,q)t^{-\frac{n}{2}\big(\frac{2}{n}-\frac{1}{q}\big)}e^{-t\beta_3(n,n,q)} \|a\|_{L^n(\mathbb{H}^n)} \\
&\leq C(n,q)t^{-1+\frac{n}{2q}}e^{-t\beta''}\|a\|_{L^n(\mathbb{H}^n)}.
\end{split}
\end{equation}

Next, if $T>0$ is chosen so that \eqref{estrequire} holds, then \eqref{deltaest}, \eqref{M0global}, and the definition of $\beta''$ imply for our chosen $\delta$ that
\begin{equation} \label{rrr}
\|u(t)\|_{L^{n/\delta}(\mathbb{H}^n)} \leq Mt^{-(1-\delta)/2}e^{-t\beta''}, \quad 0 \leq t < T.
\end{equation}
Using this fact and taking $\gamma = \zeta =n/q$ and $\alpha = \delta$ in \eqref{DGest}, it follows that
\begin{equation} \label{nDdeltacomp1}
\begin{split}
\|\nabla Gu(t)\|_{L^{q}(\mathbb{H}^n)} &\leq C(n,q,\delta) \int_0^t (t-s)^{-(1+\delta)/2}e^{-(t-s)\beta_3\big(n,\frac{n}{n/q+\delta},q\big) } \|u(s)\|_{L^{n/\delta}(\mathbb{H}^n)}\|\nabla u(s)\|_{L^{q}(\mathbb{H}^n)}\,ds \\
& \leq C(n,q,\delta)M\int_0^t(t-s)^{-(\delta+1)/2}e^{-(t-s)\beta''}s^{-(1-\delta)/2}e^{-s\beta''}\|\nabla u(s)\|_{L^{q}(\mathbb{H}^n)}\,ds \\
& = C(n,q,\delta)Me^{-t\beta''}\int_0^t(t-s)^{-(\delta+1)/2}s^{-(1-\delta)/2}\|\nabla u(s)\|_{L^{q}(\mathbb{H}^n)}\,ds.
\end{split}
\end{equation}
Applying this estimate together with \eqref{LqD0} and using that $\nabla u(t) = \nabla u_0(t) + \nabla Gu(t)$, we have
\begin{equation}
t^{1-\frac{n}{2q}}e^{t\beta''}\|\nabla u(t)\|_{L^{q}(\mathbb{H}^n)} \leq C(n,q) \|a\|_{L^n(\mathbb{H}^n)}+C(n,q,\delta)M\int_0^tg(t,s)s^{1-\frac{n}{2q}}e^{s\beta''}\|\nabla u(s)\|_{L^{q}(\mathbb{H}^n)}\,ds,
\end{equation}
where $g(t,s) = (t-s)^{-(\delta+1)/2}s^{-(1-\delta)/2}t^{1-\frac{n}{2q}}s^{-1+\frac{n}{2q}}e^{-s\beta''}$. By Gr{\"o}nwall's inequality, 
\begin{equation}
\begin{split}
t^{1-\frac{n}{2q}}e^{t\beta''}\|\nabla u(t)\|_{L^{q}(\mathbb{H}^n)}&\leq C(n,q) \|a\|_{L^n(\mathbb{H}^n)}\text{exp}\bigg(C(n,q,\delta)M\int_0^tg(t,s)\,ds \bigg) \\
&\leq C(n,q) \|a\|_{L^n(\mathbb{H}^n)}\text{exp}\bigg(C(n,q,\delta)Mt^{1-\frac{n}{2q}}\int_0^t(t-s)^{-(\delta+1)/2}s^{-\frac{(1-\delta)}{2}}s^{-1+\frac{n}{2q}}\,ds \bigg) \\
&= C(n,q) \|a\|_{L^n(\mathbb{H}^n)}\text{exp}\bigg(C(n,q,\delta)Mt^{1-\frac{n}{2q}}\int_0^t(t-s)^{-(\delta+1)/2}s^{\big(\frac{\delta}{2}+\frac{n}{2q} - \frac{1}{2}\big) - 1}\,ds \bigg) \\
&= C(n,q) \|a\|_{L^n(\mathbb{H}^n)}\text{exp}\bigg(C(n,q,\delta)Mt^{1-\frac{n}{2q}}t^{-1+\frac{n}{2q}}B\bigg(\frac{\delta}{2}+\frac{n}{2q}-\frac{1}{2}, \frac{1-\delta}{2}\bigg) \bigg) \\
&= C(n,q) \|a\|_{L^n(\mathbb{H}^n)}\text{exp}\bigg(C(n,q,\delta)MB\bigg(\frac{\delta}{2}+\frac{n}{2q}-\frac{1}{2}, \frac{1-\delta}{2}\bigg) \bigg). 
\end{split}
\end{equation}
As well, the assumption \eqref{qbdd2} implies
\[
\frac{\delta}{2}+\frac{n}{2q}-\frac{1}{2}  >0,
\]
so that $B\big(\frac{\delta}{2}+\frac{n}{2q}-\frac{1}{2}, \frac{1-\delta}{2}\big)$ is finite.

We conclude that 
\begin{equation}
\sup_{0\leq t < T} t^{1-\frac{n}{2q}}e^{t\beta''}\|\nabla u(t)\|_{L^{q}(\mathbb{H}^n)} \leq C(n,q) \|a\|_{L^n(\mathbb{H}^n)}\text{exp}\bigg(C(n,q,\delta)MB\bigg(\frac{\delta}{2}+\frac{n}{2q}-\frac{1}{2}, \frac{1-\delta}{2}\bigg) \bigg),
\end{equation}
so that $t^{1-\frac{1}{2q}}e^{t\beta''} \nabla u \in L^{\infty}([0,T), L^{q}(\mathbb{H}^n))$.

Regarding the global estimate for $\nabla u$, as argued in previous sections, if $\|a\|_{L^{n}(\mathbb{H}^n)}$ satisfies the smallness condition \eqref{asmall}, then the estimate \eqref{rrr} holds for all $0 \leq t < \infty$, and the same computations presented above for the local estimate on $\nabla u$ imply $t^{1-\frac{1}{2q}}e^{t\beta''} \nabla u \in L^{\infty}([0,\infty), L^{q}(\mathbb{H}^n))$.


 
\section{Proof of Theorem \ref{PSL}} \label{PSLpf}  We now show the solution $u$ found in the previous section satisfies Theorem \ref{PSL}. Following closely the methods of Giga in \cite{giga_semilinear}, the first step is to show by induction that for some $T_1>0$ to be chosen later, the sequence \eqref{iterate} satisfies
\begin{equation} \label{mixinduc}
u_k \in L^{r}\big((0,T_1), L^q(\mathbb{H}^n)\big) \quad\quad \text{with } \quad \frac{1}{r} = \frac{1}{2} - \frac{n}{2q}, \quad n < q < \frac{n^2}{n-2}.
\end{equation}
For the case $k=0$, fix $r$ and $n< q$ such that 
\begin{equation} \label{u0mixest}
\frac{1}{r} = \bigg(\frac{1}{n} - \frac{1}{q}\bigg) \frac{n}{2}
\end{equation}
and define a map $U$ from $L^q(\mathbb{H}^n)$ to functions on $(0,T_1)$ by
\[
Uf = \|e^{tL}f\|_{L^q(\mathbb{H}^n)}.
\]
Let $\tilde{n} = n- \varepsilon$, where $\varepsilon >0$ is to be chosen below and such that $n-\varepsilon >0$.  As well, define $\tilde{r}$ by 
\begin{equation} \label{u0mixest1}
\frac{1}{\tilde{r}} = \bigg(\frac{1}{\tilde{n}} - \frac{1}{q}\bigg) \frac{n}{2}.
\end{equation}
We claim that $U$ is of weak type $(\tilde{n},\tilde{r})$. To see this, we must show
\begin{equation} \label{measure}
m\{\tau : |Uf(\tau)| > t\} \leq \bigg(\frac{C\|f\|_{L^{\tilde{n}}(\mathbb{H}^n)}}{t}\bigg)^{\tilde{r}}, \quad \text{for every } t>0.
\end{equation}
Let $\tau \in \{\tau : |Uf(\tau)| > t\}$. Then by \eqref{dispest} with $\tilde{n}=p$, we have
\begin{align*}
t &< |Uf(\tau)| \\
&=\|e^{\tau L}f\|_{L^q(\mathbb{H}^n)} \\
&\leq C(n, \varepsilon,q) \tau^{-\big(\frac{1}{\tilde{n}}- \frac{1}{q}\big)\frac{n}{2}}\|f\|_{L^{\tilde{n}}(\mathbb{H}^n)} \\
&=C(n,\varepsilon,q) \tau^{-\frac{1}{\tilde{r}}} \|f\|_{L^{\tilde{n}}(\mathbb{H}^n)},
\end{align*}
which gives
\[
\tau < \bigg(\frac{C(n,\varepsilon,q)\|f\|_{L^{\tilde{n}}(\mathbb{H}^n)}}{t}\bigg)^{\tilde{r}}.
\]
From this it follows that
\[
\{\tau : |Uf(\tau)| > t\} \subset \bigg[0, \bigg(\frac{C(n,\varepsilon, q)\|f\|_{L^{\tilde{n}}(\mathbb{H}^n)}}{t}\bigg)^{\tilde{r}}\bigg]
\]
and thus \eqref{measure} is verified. We also get that $U$ is of weak type $(q, \infty)$ by taking $p=q$ in \eqref{dispest}, which gives 
\begin{align*}
|Uf(t)| &= \|e^{tL}f\|_{L^q(\mathbb{H}^n)} \\
& \leq C(n,q) \|f\|_{L^q(\mathbb{H}^n)}. 
\end{align*}

Next we note that if $ q < \dfrac{n^2}{n-2}$, then $0 < n - \dfrac{q(n-2)}{n}$. Thus if we choose $\varepsilon$ so that
\[
\varepsilon < n - \dfrac{q(n-2)}{n},
\]
then $n - \varepsilon > \dfrac{q(n-2)}{n} \geq0$, and it follows that $\tilde{n} < \tilde{r}$. Having shown that $U$ is of weak type $(\tilde{n},\tilde{r})$ and $(q, \infty)$, noting that $U$ is subadditive by Minkowski's integral inequality, and using that $\tilde{n} < n< q$ and $\tilde{n} < \tilde{r}$, we can then apply the Marcinkiewicz interpolation theorem as follows (see, for example, \cite{Stein}): For $0 < \theta < 1$, $U$ is of strong-type $(n_1,r_1)$, where
\begin{equation} \label{interpineq}
\frac{1}{n_1} = \frac{1-\theta}{\tilde{n}}+\frac{\theta}{q}, \quad \frac{1}{r_1}=\frac{1-\theta}{\tilde{r}}.
\end{equation}
Solving for $(1-\theta)$ in the first equation gives
\begin{equation} \label{1-theta}
1-\theta = \bigg(\frac{1}{q}-\frac{1}{n_1}\bigg)\bigg(\frac{1}{q}-\frac{1}{\tilde{n}}\bigg)^{-1},
\end{equation}
and combining this with the fact that
\[
\frac{1}{\tilde{r}} = \bigg(\frac{1}{\tilde{n}}-\frac{1}{q}\bigg)\frac{n}{2},
\]
the second equation in \eqref{interpineq} can be rewritten as
\begin{equation} \label{interp}
\frac{1}{r_1} = \bigg(\frac{1}{n_1} - \frac{1}{q}\bigg)\frac{n}{2}. 
\end{equation}
By definition, the pair $(n,r)$ satisfies \eqref{interp}. Moreoever, since $\tilde{n} < n < q$ is assumed and since
\begin{equation}
\frac{1}{\tilde{r}} = \bigg(\frac{1}{n-\varepsilon}-\frac{1}{q}\bigg)\frac{n}{2} > \bigg(\frac{1}{n}-\frac{1}{q}\bigg)\frac{n}{2} = \frac{1}{r},
\end{equation}
we have $\tilde{r} < r < \infty$. From this we conclude $U$ is of strong type $(n,r)$. 

With this fact estblished, we can then estimate the $L^{r}\big((0,T_1), L^q(\mathbb{H}^n)\big)$ norm of $u_0$ in the following way:
\begin{equation} \label{u0a}
\begin{split}
\|u_0(t)\|_{L^{r}((0,T_1), L^q(\mathbb{H}^n))} &= \bigg[\int_0^{T_1}\|e^{tL}a\|_{L^q(\mathbb{H}^n)}^r\,dt\bigg]^{1/r} \\
&= \|Ua\|_{L^r((0,T_1))} \\
&\leq C(n,q)\|a\|_{L^n(\mathbb{H}^n)},
\end{split}
\end{equation}
which gives $u_0 \in L^{r}\big((0,T_1), L^q(\mathbb{H}^n)\big)$. 

Assuming now that \eqref{mixinduc} is true for $k$, we prove it for for $k+1$. To do this, we will use the following simple corollary to the one-dimensional form of the Hardy-Littlewood-Sobolev lemma (see \cite{Stein}, for example).

\begin{lem} \label{HLSmod} 
Let $ 0 < T_1$, $0 < \eta < 1$, and $1 < \lambda < \mu < \infty$ such that
\begin{equation} \label{HLSineq}
\frac{1}{\mu} = \frac{1}{\lambda} - \eta
\end{equation}
and define $I_{\eta}$ by
\begin{equation} \label{singINT}
I_{\eta}(f)(x) = C(\eta) \int_0^{T_1} \frac{f(s)}{|t-s|^{1-\eta}}\,ds.
\end{equation}
Then
\begin{equation} \label{HLS}
\|I_{\eta}(f)\|_{L^{\mu}((0,T_1))} \leq C(\mu, \lambda) \|f\|_{L^{\lambda}((0,T_1))}. 
\end{equation}
\end{lem}
\begin{proof}
Letting $\chi_{[0,T_1]}$ denote the indicator function for the interval $[0, T_1]$, we have by the Hardy-Littlewood-Sobolev lemma that
\begin{equation}
\begin{split}
\|I_{\eta}(f)\|_{L^{\mu}((0,T_1))} &= \bigg\|C(\eta)\int_{\R}\frac{f(s)\chi_{[0,T_1]}(s)}{|t-s|^{1-\eta}}\,ds. \bigg\|_{L^{\mu}(\R)}\\
& \leq C(\mu, \lambda) \|f\chi_{[0,T_1]}\|_{L^{\lambda}(\R)} \\
&=C(\mu, \lambda) \|f\|_{L^{\lambda}((0,T_1))}.
\end{split}
\end{equation}
\end{proof}
\noindent To apply this lemma, we take $\alpha= \zeta = \gamma = n/q$ in \eqref{Gest} to get
\begin{equation}
\begin{split}
\|Gu_k(t)\|_{L^q(\mathbb{H}^n)} &\leq C(n,q)\int_0^t (t-s)^{-\big(\frac{1}{2}+\frac{n}{2q}\big)}\big(\|u_k(s)\|_{L^q(\mathbb{H}^n)}\big)^2\,ds \\
&\leq \frac{C(n,q)C(1/2-n/2q)}{C(1/2-n/2q)} \int_{0}^{T_1} \frac{\big(\|u_k(s)\|_{L^q(\mathbb{H}^n)}\big)^2}{(t-s)^{1-\big(\frac{1}{2}-\frac{n}{2q}\big)}}\, ds \\
&= C(n,q) I_{(1/2-n/2q)}\big(\|u_k(\cdot)\|_{L^q(\mathbb{H}^n)}^2\big)(t).
\end{split}
\end{equation}
Taking the preceding inequality and applying the $L^{r}((0, T_1))$ norm as well as \eqref{HLS}, we get
\begin{equation} \label{Gupartial}
\begin{split}
\|Gu_k\|_{L^r((0, T_1), L^q(\mathbb{H}^n))} &\leq C(n,q) \big\|I_{(1/2-n/2q)}\big(\|u_k(\cdot)\|_{L^q(\mathbb{H}^n)}^2\big)\big\|_{L^r((0,T_1))} \\
&\leq C(n,q) \big\| \|u_k(\cdot)\|_{L^q(\mathbb{H}^n)}^2\big\|_{L^{\lambda}((0,T_1))},
\end{split}
\end{equation}
where
\begin{equation}
\frac{1}{\lambda} = \frac{1}{r}+\bigg(\frac{1}{2}-\frac{n}{2q}\bigg) = \frac{2}{r}
\end{equation}
by \eqref{u0mixest}. Thus \eqref{Gupartial} becomes 
\begin{equation} 
\begin{split}
\|Gu_k\|_{L^r((0, T_1), L^q(\mathbb{H}^n))} &\leq C(n,q) \big\| \|u_k(\cdot)\|_{L^q(\mathbb{H}^n)}^2\big\|_{L^{r/2}((0,T_1))} \\
&=C(n,q)\|u_k\|_{L^r((0, T_1), L^q(\mathbb{H}^n))}^2.
\end{split}
\end{equation}
Using \eqref{iterate}, we so far have shown
\begin{equation}
\|u_{k+1}\|_{L^r((0, T_1), L^q(\mathbb{H}^n))}  \leq \|u_0\|_{L^r((0, T_1), L^q(\mathbb{H}^n))} +C(n,q)\|u_k\|_{L^r((0, T_1), L^q(\mathbb{H}^n))}^2,
\end{equation}
which is structurally the same recurrence inequality as \eqref{kplus1}. This has been shown to define a bounded sequence as long as
\begin{equation} \label{u0reqest}
\|u_0\|_{L^r((0, T_1), L^q(\mathbb{H}^n))} < \frac{1}{4C(n,q)},
\end{equation}
and since $\|u_0\|_{L^r((0, T_1), L^q(\mathbb{H}^n))}$ is bounded as shown in \eqref{u0a}, we can choose $T_1\leq T$ sufficiently small so that \eqref{u0reqest} holds. As well, since our sequence $u_k \in {L^r((0, T_1), L^q(\mathbb{H}^n))}$ and since $u_k$ converges to $u$ in $L^{\infty}([0,T), L^q(\mathbb{H}^n))$ by \eqref{PFthm}, we conclude $u \in {L^r((0, T_1), L^q(\mathbb{H}^n))}$.

As for the global in time result, by \eqref{u0a} it follows that if 
\begin{equation}
\|a\|_{L^n(\mathbb{H}^n)} < \frac{1}{4C(n,q)},
\end{equation}
where this constant $C(n,q)$ is the product of the constants appearing in \eqref{Gupartial} and \eqref{u0a}, respectively, then 
\begin{equation}
\|u_0\|_{L^r((0, T_1), L^q(\mathbb{H}^n))} < \frac{1}{4C(n,q)},
\end{equation}
for any choice of $T_1> 0$ and in this case, we conclude $u \in {L^r((0, \infty), L^q(\mathbb{H}^n))}$.


\section{Proof of Theorem \ref{newLrLq}} \label{newLrLqpf} Suppose $u$ is the mild solution from Theorem \ref{EUthm}, with corresponding time of existence $T>0$. Then for $n < q$ and $0 \leq t < T$,
\begin{equation} \label{LrLq1}
\|u(t)\|_{L^q(\mathbb{H}^n)} \leq Mt^{-\big(\frac{1}{2} - \frac{n}{2q}\big)}e^{-t\beta},
\end{equation}
where $M$ is defined by \eqref{M0global} (with $\delta = n/q$), and where $\beta >0$ is defined by \eqref{1stbeta}. Thus for $1 \leq r < \infty$ such that
\be \label{newPSL2}
\frac{1}{2} - \frac{n}{2q} < \frac{1}{r},
\ee
and for $0 < T \leq 1$, we get that
\be
\begin{split}
\|u\|^r_{L^r((0,T), L^q(\H^n))} &= \int_0^T \|u(t)\|^r_{L^q(\H^n)}\, dt \\
&\leq C(M,r) \int_0^1 t^{-r\big(\frac{1}{2} - \frac{n}{2q}\big)}e^{-tr\beta}\,dt \\
&\leq C(M, r) \int_0^1 t^{-r\big(\frac{1}{2} - \frac{n}{2q}\big)}\,dt \\
& \leq C(M,n,q,r).
\end{split}
\ee
This shows Theorem \ref{newLrLq}, at least locally, and the same proof would work for $\R^n$ as well.

However, for the case of large time of existence $T >1$, we exploit the exponential decay present in \eqref{LrLq1}, which is not available in the Euclidean setting, and which allows us to gain uniform control for $1 < t < \infty$. Specifically, we have
\be
\begin{split}
\|u\|^r_{L^r((0,T), L^q(\H^n))} &= \int_0^T \|u(t)\|^r_{L^q(\H^n)}\, dt \\
&= \int_0^1 \|u(t)\|^r_{L^q(\H^n)}\, dt + \int_1^T \|u(t)\|^r_{L^q(\H^n)}\, dt \\
&\leq C\big(M,r)\bigg( \int_0^1 t^{-r\big(\frac{1}{2} - \frac{n}{2q}\big)}\,dt +\int_1^T e^{-tr\beta}\,dt\bigg) \\
&\leq C\big(M,r)\bigg( \int_0^1 t^{-r\big(\frac{1}{2} - \frac{n}{2q}\big)}\,dt +\int_1^T e^{-tr\beta}\,dt\bigg) \\
& = C(M,n,q,r,\beta),
\end{split}
\ee
where in the last step, we have used that
\be \label{unfmest}
\begin{split}
\int_1^T e^{-tr\beta}\, dt &= C(r, \beta) (e^{-r\beta} - e^{-Tr\beta}) \\
& \leq C(r, \beta) (e^{-r\beta}).
\end{split}
\ee
Therefore, for the case when $\|a\|_{L^n(\H^n)}$ is small, so that $u$ is global and estimate \eqref{LrLq1} holds for all $0 \leq t < \infty$, 
\be
\|u\|_{L^r((0,\infty), L^q(\H^n))} \leq C(M,n,q,r,\beta),
\ee
and as we have shown in \eqref{unfmest}, this bound is uniform in time. Hence $u \in L^r((0,\infty), L^q(\H^n))$.

\section{Proof of Theorem \ref{LrLn}} \label{LrLnpf} As in the previous section, we suppose $u$ is the mild solution from Theorem \ref{EUthm}, with corresponding time of existence $T>0$. Then for $0 \leq t < T$,
\begin{equation} \label{LrLn1}
\|u(t)\|_{L^n(\mathbb{H}^n)} \leq C\big(n, \delta, \|a\|_{L^n(\mathbb{H}^n)}\big)e^{-t\beta'},
\end{equation}
where for a fixed $0 < \delta < 1$, $C(n, \delta, \|a\|_{L^n(\mathbb{H}^n)})$ is the constant appearing in \eqref{mmm}, and $\beta' >0$ is defined by \eqref{2ndbeta}. With this, we can estimate the $L^r((0,\infty), L^n(\H^n))$ norm of $u$ for $1 \leq r < \infty$ as follows:
\be
\begin{split}
\|u\|^r_{L^r((0,T), L^q(\H^n))} &\leq C\big(n, r, \delta, \|a\|_{L^n(\mathbb{H}^n)}\big) \int_0^Te^{-tr\beta'}\,dt \\
&= C\big(n, r, \beta', \delta, \|a\|_{L^n(\mathbb{H}^n)}\big)(1 - e^{-Tr\beta'}) \\
&\leq C\big(n, r, \beta', \delta, \|a\|_{L^n(\mathbb{H}^n)}\big).
\end{split}
\ee
Thus $u \in L^r((0,T), L^n(\H^n))$. Moreover, since this estimate is uniform in time, if $\|a\|_{L^n(\H^n)}$ is small, so that the solution $u$ is global, and so that estimate \eqref{LrLn1} holds for all $0 \leq t < \infty$, then $u \in L^r((0,\infty), L^n(\H^n))$.

\section{Proof of Theorem \ref{tmdcy1}} \label{tmdcy1pf}
In the situation from Theorem \ref{EUthm} where the solution $u$ is global, we have the following estimate:
\begin{equation}
\|u(t)\|_{L^n(\mathbb{H}^n)} \leq C\big(n, \delta, \|a\|_{L^n(\mathbb{H}^n)}\big)e^{-t\beta'},
\end{equation}
where for a fixed $0 < \delta < 1$, $C(n, \delta, \|a\|_{L^n(\mathbb{H}^n)})$ is the constant appearing in \eqref{mmm}, and $\beta' >0$ is defined by \eqref{2ndbeta}. Therefore,
\begin{equation}
\begin{split}
\lim_{T \to \infty} \frac{1}{T}\int_0^T \|u(t)\|_{L^n(\mathbb{H}^n)}\, dt & \leq \lim_{T \to \infty} \frac{C\big(n, \delta, \|a\|_{L^n(\mathbb{H}^n)}\big)}{T} \int_0^T e^{-\beta' t}\, dt\\
&\leq \lim_{T \to \infty} \frac{C\big(n, \delta, \|a\|_{L^n(\mathbb{H}^n)}\big)}{\beta' T}\big(1-e^{-\beta' T}\big) \\
& = 0.
\end{split}
\end{equation}
This proves \eqref{tmint}.

The proof for \eqref{Lndecay} is even more straightforward. As above, we have 
\begin{equation}
\|u(t)\|_{L^n(\mathbb{H}^n)} \leq C\big(n, \delta, \|a\|_{L^n(\mathbb{H}^n)}\big)e^{-t\beta'},
\end{equation}
so that
\begin{equation}
\begin{split}
\lim_{t \to \infty}\|u(t)\|_{L^n(\mathbb{H}^n)} &\leq \lim_{t \to \infty}C\big(n, \delta, \|a\|_{L^n(\mathbb{H}^n)}\big)e^{-t\beta'} \\
&=0.
\end{split}
\end{equation}


\section{Proof of Theorem \ref{Lpthm}} \label{Lpthmpf} 
\subsection{$L^p$ estimates for $u$}
If $a \in L^n(\mathbb{H}^n) \cap L^p(\mathbb{H}^n)$ for $1 < p < n$, then a unique solution $u$ exists for some $T>0$ and such that $u$ satisfies \eqref{mainspace} with $q=n$. Thus to prove Theorem \ref{Lpthm}, we first choose $0 < \delta < 1$ such that $n/p+\delta < n$. With this $\delta$ chosen, we will show
\begin{equation}
e^{t\beta^*} u \in BC\big([0,T), L^p(\mathbb{H}^n) \cap L^n(\mathbb{H}^n)\big),
\end{equation}
where 
\begin{equation} \label{betastar}
\beta^* = \min\bigg\{\beta, \beta_1(n, p,p), \beta_3\bigg(n, \frac{n}{n/p+\delta},p\bigg) \bigg\},
\end{equation}
and where $\beta>0$ is defined by \eqref{1stbeta}.

To estimate the free solution $u_0(t)$, we set $q=p$ in \eqref{dispest}, which gives
\begin{equation} \label{u0lp}
\begin{split}
\|u_0(t)\|_{L^p(\mathbb{H}^n)} &\leq C(n,p)e^{-t\beta_1(n,p,p)} \|a\|_{L^p(\mathbb{H}^n)} \\
&\leq C(n,p)e^{-t\beta^*} \|a\|_{L^p(\mathbb{H}^n)} \\
\end{split}
\end{equation}

For estimating $Gu(t)$, we first observe that if $T$ is chosen so that \eqref{estrequire} holds, then for our chosen $0 < \delta < 1$ and any $0 \leq t < T$,
\begin{equation}
\|u(t)\|_{L^{n/\delta}(\mathbb{H}^n)} \leq M t^{-(1-\delta)/2}e^{-t\beta^*},
\end{equation}
where we have used that $e^{-t\beta} \leq e^{-t\beta^*}$. Applying \eqref{Gest} with $\alpha = \gamma = \frac{n}{p}$ and $\zeta = \delta$, and noting that our choice of $\delta$ is such that $n/p + \delta < n$, we have
\begin{equation} \label{Lnest} 
\begin{split} 
\|Gu(t)\|_{L^p(\mathbb{H}^n)} &\leq C(n,p,\delta)\int_0^t (t-s)^{-\frac{(\delta+1)}{2}}e^{-(t-s)\beta_3\big(n, \frac{n}{n/p+\delta},p\big)} \|u(s)\|_{L^{n/\delta}(\mathbb{H}^n)}\|u(s)\|_{L^{p}(\mathbb{H}^n)}\,ds \\
& \leq C(n,p,\delta)M\int_0^t (t-s)^{-\frac{(\delta+1)}{2}}e^{-(t-s)\beta^*}s^{-\frac{(1-\delta)}{2}}e^{-s\beta^*} \|u(s)\|_{L^{p}(\mathbb{H}^n)}\,ds \\
& = C(n,p,\delta)Me^{-t\beta^*}\int_0^t (t-s)^{-\frac{(\delta+1)}{2}}s^{-\frac{(1-\delta)}{2}} \|u(s)\|_{L^{p}(\mathbb{H}^n)}\,ds.
\end{split}
\end{equation}
Taking this estimate together with \eqref{u0lp}, using that $u(t) = u_0(t) + Gu(t)$, and multiplying through by $e^{t\beta^*}$, we have shown that
\begin{equation}
e^{t\beta^*}\|u(t)\|_{L^p(\mathbb{H}^n)} \leq C(n,p)\|a\|_{L^{n}(\mathbb{H}^n)} + C(n,p,\delta)M\int_0^t (t-s)^{-\frac{(\delta+1)}{2}}s^{-\frac{(1-\delta)}{2}}e^{-s\beta^*}e^{s\beta^*} \|u(s)\|_{L^{p}(\mathbb{H}^n)}\,ds,
\end{equation}
so that by Gr{\"o}nwall's inequality,
\begin{equation}
\begin{split}
e^{t\beta^*}\|u(t)\|_{L^p(\mathbb{H}^n)} &\leq C(n,p)\|a\|_{L^{n}(\mathbb{H}^n)}\text{exp}\bigg( C(n,p,\delta)M\int_0^t (t-s)^{-\frac{(\delta+1)}{2}}s^{-\frac{(1-\delta)}{2}}e^{-s\beta^*} \,ds\bigg) \\
&\leq C(n,p)\|a\|_{L^{n}(\mathbb{H}^n)}\text{exp}\bigg( C(n,p,\delta)MB\bigg(\frac{\delta+1}{2}, \frac{1-\delta}{2}\bigg)\bigg),
\end{split}
\end{equation}
where we have made the substitution $\tau = s/t$ to compute the Beta integral. This shows
\begin{equation}\label{nnn}
\sup_{0 \leq t < T}e^{t\beta^*}\|u(t)\|_{L^p(\mathbb{H}^n)} \leq C(n,p)\|a\|_{L^{n}(\mathbb{H}^n)}\text{exp}\bigg( C(n,p,\delta)MB\bigg(\frac{\delta+1}{2}, \frac{1-\delta}{2}\bigg)\bigg),
\end{equation}
so that $e^{t\beta^*}u \in L^{\infty}([0,T), L^p(\mathbb{H}^n))$. 

For the global result, if $\|a\|_{L^{n}(\mathbb{H}^n)}$ satisfies the smallness condition \eqref{asmall}, then as shown in the proof of \eqref{mainspace} from Theorem \ref{EUthm},
\begin{equation} 
\|u(t)\|_{L^{n/\delta}(\mathbb{H}^n)} < Mt^{-\big(\frac{1}{2} - \frac{\delta}{2}\big)}e^{-t\beta^*}, \quad \text{for } 0 \leq t < \infty,
\end{equation}
where we recall the definition of $M$ as 
\begin{equation}
M = \frac{1}{2C(n, \delta)},
\end{equation}
with $C(n, \delta)$ being the constant appearing in the recurrence inequality \eqref{kplus1}. Thus the steps used to derive estimate \eqref{nnn} are valid for $0 \leq t <\infty$ and we conclude $e^{t\beta^*}u \in L^{\infty}([0,\infty), L^p(\mathbb{H}^n))$, as desired.

\subsection{Continuous dependence on initial data} Let $a \in L^p(\mathbb{H}^n) \cap L^n(\mathbb{H}^n)$ and let $u$ be the corresponding solution from Theorem \ref{Lpthm} with time of existence $T$. As well, let $a' \in L^p(\mathbb{H}^n) \cap L^n(\mathbb{H}^n)$ and let $v$ be the corresponding solution from Theorem \ref{Lpthm} with time of existence $T'$. Let $\tilde{T} = \text{min}\{T, T'\}.$

Let $0 < \delta < 1$ be such that $n/p+\delta < n$. By \eqref{deltaest}, both $u$ and $v$ satisfy
\begin{align}
\sup_{0\leq t < \tilde{T}}t^{\frac{(1-\delta)}{2}}e^{t\beta^*}\|u\|_{L^{n/\delta}(\mathbb{H}^n)} &\leq \tilde{M}, \\
\sup_{0\leq t < \tilde{T}}t^{\frac{(1-\delta)}{2}}e^{t\beta^*}\|v\|_{L^{n/\delta}(\mathbb{H}^n)} &\leq \tilde{M},
\end{align}
where, letting $M_u$ denote the uniform bound on the sequence $\{u_k\}$ from Theorem \ref{EUthm} and $M_v$ that for $\{v_k\}$, we define $\tilde{M} = \text{max}\{M_u, M_v\}$. Using these facts and applying \eqref{Gdest} with $\alpha = \delta$ and $\zeta = \gamma = n/p$ as well as \eqref{dispest} with $p=q$, we have
\begin{equation}
\begin{split}
\|u(t) - v(t) \|_{L^p} &\leq \|e^{tL}(a-a')\|_{L^p}\\ &\quad + C(n, p, \delta)\int_0^t (t-s)^{-\frac{(\delta+1)}{2}}e^{-(t-s)\beta_3\big(n, \frac{n}{n/p+\delta},p\big)} (\|u\|_{L^{n/\delta}} +\|v\|_{L^{n/\delta}})\|u-v\|_{L^{p}}\,ds \\
& \leq C(n,p)e^{-t\beta^*}\|a-a'\|_{L^p}\\ &\quad + 2C(n, p, \delta)\tilde{M}\int_0^t (t-s)^{-\frac{(\delta+1)}{2}}e^{-(t-s)\beta^*}s^{-\frac{(1-\delta)}{2}}e^{-s\beta^*} \|u-v\|_{L^{p}}\,ds \\
& \leq C(n,p)e^{-t\beta^*}\|a-a'\|_{L^p}\\ &\quad + 2C(n, p, \delta)\tilde{M}e^{-t\beta^*}\int_0^t (t-s)^{-\frac{(\delta+1)}{2}}s^{-\frac{(1-\delta)}{2}}e^{-s\beta^*}e^{s\beta^*} \|u-v\|_{L^{p}}\,ds.
\end{split}
\end{equation}
Multiplying the above computation through by $e^{t\beta^*}$ and applying Gr{\"on}wall's inequality,
\begin{equation}
\begin{split}
e^{t\beta^*}\|u(t) - v(t) \|_{L^p(\mathbb{H}^n)} &\leq C(n,p)\|a-a'\|_{L^p(\mathbb{H}^n)}\text{exp}\bigg(2C(n, p, \delta)\tilde{M}\int_0^t (t-s)^{-\frac{(\delta+1)}{2}}s^{-\frac{(1-\delta)}{2}}e^{-s\beta^*}\,ds\bigg) \\
& =C(n,p)\|a-a'\|_{L^p(\mathbb{H}^n)}\text{exp}\bigg(2C(n, p, \delta)\tilde{M}B\bigg(\frac{\delta+1}{2}, \frac{1-\delta}{2}\bigg)\bigg) \\
& = C(n, p, \delta, \tilde{M})\|a-a'\|_{L^p(\mathbb{H}^n)}.
\end{split}
\end{equation}
Therefore, given $\varepsilon >0$, if
\[
\|a-a'\|_{L^p(\mathbb{H}^n)} < \frac{\varepsilon}{ C(n, p, \delta, \tilde{M})},
\]
then 
\[
e^{t\beta^*}\|u(t) - v(t) \|_{L^p(\mathbb{H}^n)} < \varepsilon.
\]
This shows continuous dependence on the initial data in the space $L^{\infty}\big([0,\tilde{T}), L^p(\mathbb{H}^n))$ and since continuous dependence on the initial data was already shown for $L^{\infty}\big([0,\tilde{T}), L^n(\mathbb{H}^n))$ in Theorem \ref{EUthm}, we conclude it holds in $L^{\infty}\big([0,\tilde{T}), L^p(\mathbb{H}^n) \cap L^n(\mathbb{H}^n) \big)$ as well.


\subsection{$L^p$ estimates for $\nabla u$} The proof for showing \eqref{lpDspace} proceeds almost exactly as in the proof for \eqref{mainDspace} in Theorem \ref{EUthm}, though here we choose $q=p$ in \eqref{smest2} to estimate $\nabla u_0$ and to estimate $\nabla Gu$, we apply \eqref{deltaest}, \eqref{M0global}, and take $\gamma = \zeta = n/p$ and $\alpha = \delta$ in \eqref{DGest}, where  $0 < \delta < 1$ is chosen so that $\alpha + \zeta = \delta + n/p < n$. The constant appearing in the exponential term in this case is
\begin{equation}
\beta^{**} = \min\bigg\{\beta, \beta_3(n, p,p), \beta_3\bigg(n, \frac{n}{n/p+\delta},p\bigg) \bigg\}.
\end{equation}


\section{Proof of Theorem \ref{LrLp}} \label{LrLppf} The proof of Theorem \ref{LrLp} is almost identical to that of Theorem \ref{LrLn} as presented in Section \ref{LrLnpf}, and may be omitted.


\section{Proof of Theorem \ref{tmdcy2}} \label{tmdcy2pf}

Assuming the solution $u$ and its derivative $\nabla u$ from Theorem \ref{Lpthm} are global, we first prove \eqref{td3} and in order to do this, we must prove a modified version of estimate \eqref{Gest} appearing in Lemma \ref{gtermest}. Using that $d^*u=0$ to write
\begin{equation}
Gu(t) = -\int_0^te^{-(t-s)L}\mathbb{P }\,\big(\nabla_{u^{\#}} u)(s)\,ds
\end{equation}
and letting $1 < r \leq q < \infty$, we have by our dispersive estimate \eqref{dispest}, the $L^r$ boundedness of $\mathbb{P}$ shown in Corollary \ref{Pbdd}, the pointwise estimate \eqref{covarest}, and H{\"o}lder's inequality that
\begin{equation}
\begin{split}
\|Gu(t)\|_{L^q(\mathbb{H}^n)} &\leq C(n,r, q) \int_0^t (t-s)^{-\frac{n}{2}\big(\frac{1}{r} - \frac{1}{q}\big)}e^{-(t-s)\beta_1(n,r,q)} \big\||\nabla u(s)| |u(s)| \big\|_{L^r(\mathbb{H}^n)} \,ds \\
&\leq C(n,\mu, \lambda, q) \int_0^t (t-s)^{-\frac{n}{2}\big(\frac{1}{\mu} +\frac{1}{\lambda} - \frac{1}{q}\big)}e^{-(t-s)\beta_1\big(n,\frac{\mu\lambda}{\mu+\lambda},q\big)} \|u(s) \|_{L^{\mu}(\mathbb{H}^n)}\|\nabla u(s) \|_{L^{\lambda}(\mathbb{H}^n)} \,ds,
\end{split}
\end{equation}
where $\mu, \lambda >0$ and $\frac{1}{\mu} +\frac{1}{\lambda} = \frac{1}{r}$. Then, setting
\[
q = \frac{n}{\gamma}, \quad \mu = \frac{n}{\alpha}, \quad \lambda = \frac{n}{\zeta},
\]
where $\gamma \leq \alpha + \zeta < n$, we get
\begin{equation} \label{GGest}
\|Gu(t)\|_{L^{n/\gamma}(\mathbb{H}^n)} \leq C(n, \alpha, \gamma, \zeta)\int_0^t (t-s)^{-\frac{\alpha +\zeta- \gamma}{2}}e^{-(t-s)\beta_1\big(n,\frac{n}{\alpha+\zeta},\frac{n}{\gamma}\big)} \|u(s) \|_{L^{n/\alpha}(\mathbb{H}^n)}\|\nabla u(s) \|_{L^{n/\zeta}(\mathbb{H}^n)} \,ds.
\end{equation}

With the estimate \eqref{GGest} established, we fix $q$ such that $p \leq q < \infty$ and first suppose 
\[
\frac{n}{2p}-\frac{n}{2q} < 1.
\]
Chosing $0 <\delta <1 $ such 
\begin{equation} \label{deltassump}
\frac{n}{p}-\frac{n}{q} +\delta< 2,
\end{equation}
we set
\begin{equation} \label{betatilde}
\tilde{\beta} = \min\bigg\{\beta, \beta^*,\beta'', \beta_1(n,p,q), \beta_1\bigg(n,\frac{n}{n/p+\delta},q\bigg)\bigg\}, 
\end{equation}
where for $0 < \delta' < 1$ chosen so that $1-\delta' < \delta$, and for $0<\delta^*<1$ chosen so that $n/p + \delta^* < n$, we let
\begin{align}
\beta &= \min\bigg\{ \beta_1\bigg(n,n,\frac{n}{\delta'}\bigg),  \beta_3\bigg(n,\frac{n}{2\delta'},\frac{n}{\delta'}\bigg)\bigg\}, \\
\beta'' &= \min\bigg\{\beta, \beta_3(n,n,q), \beta_3\bigg(n,\frac{n}{\delta+\delta'},\frac{n}{\delta'}\bigg) \bigg\}, \\
\beta^* &= \min\bigg\{\beta, \beta_1(n, p,p), \beta_3\bigg(n, \frac{n}{n/p+\delta^*},p\bigg) \bigg\}.
\end{align}

It follows by \eqref{dispest} and the definition of $\tilde{\beta}$ that
\begin{equation} \label{u0td3}
\begin{split}
t^{\frac{n}{2}\big(\frac{1}{p}- \frac{1}{q}\big)}e^{t\tilde{\beta}}\big\|u_0(t)\big\|_{L^q(\mathbb{H}^n)} &\leq t^{\frac{n}{2}\big(\frac{1}{p}- \frac{1}{q}\big)}e^{t \beta_1(n,p,q)}\big\|u_0(t)\big\|_{L^q(\mathbb{H}^n)} \\
& \leq C(n,p,q)\,\|a\|_{L^p(\mathbb{H}^n)},
\end{split}
\end{equation}
This proves the desired decay rate for $u_0$. 

For estimating $Gu$, observe that by taking $q = n/\delta$ in \eqref{mainDspace} from Theorem \ref{EUthm},
\begin{equation} \label{yyyy}
\|\nabla u(t)\|_{L^{n/\delta}(\mathbb{H}^n)} \leq C(n,\delta) t^{-1+\delta/2}e^{-t\beta''},
\end{equation}
and this estimate is justified, for in setting $q=n/\delta$, our assumption $1-\delta'<\delta$  is equivalent to the requirement $q < \frac{n}{1-\delta'}$.  

As well, by Theorem \ref{Lpthm} and our assumption that $u$ is global in time,
\begin{equation}
e^{t\tilde{\beta}}\|u(t)\|_{L^p(\mathbb{H}^n)}\leq e^{t\beta^*}\|u(t)\|_{L^p(\mathbb{H}^n)}
\end{equation} 
is bounded for $0\leq t <\infty$. Thus, setting 
\[
\alpha = \frac{n}{p}, \quad \gamma = \frac{n}{q}, \quad \zeta = \delta, 
\]
in \eqref{GGest}, we have 
\begin{equation} \label{Gtd3}
\begin{split}
\|Gu(t)\|_{L^q(\mathbb{H}^n)} &\leq C(n, p,q, \delta) \int_0^t (t-s)^{-\frac{1}{2}\big(\frac{n}{p} + \delta -\frac{n}{q}\big)}e^{-(t-s)\beta_1\big(n,\frac{n}{n/p+\delta},q\big)} \|u(s)\|_{L^{p}(\mathbb{H}^n)}\|\nabla u(s)\|_{L^{n/\delta}(\mathbb{H}^n)}\,ds\\
&\leq C(n, p,q, \delta,\delta',\delta^*) \int_0^t (t-s)^{-\frac{1}{2}\big(\frac{n}{p} + \delta -\frac{n}{q}\big)}e^{-(t-s)\tilde{\beta}}e^{-s\beta^*}s^{-1+\delta/2}e^{-s\beta''}\,ds \\
&\leq C(n, p,q, \delta,\delta',\delta^*) e^{-t\tilde{\beta}}\int_0^t (t-s)^{-\frac{1}{2}\big(\frac{n}{p} + \delta -\frac{n}{q}\big)}e^{s\tilde{\beta}}e^{-s\tilde{\beta}}s^{-1+\delta/2}e^{-s\tilde{\beta}}\,ds \\
&\leq C(n, p,q, \delta,\delta',\delta^*) e^{-t\tilde{\beta}}\int_0^t (t-s)^{-\frac{1}{2}\big(\frac{n}{p} + \delta -\frac{n}{q}\big)}s^{-1+\delta/2}\,ds \\
& \leq C(n, p,q, \delta,\delta',\delta^*)  t^{-\big(\frac{n}{2p}-\frac{n}{2q}\big)}e^{-t\tilde{\beta}} \int_0^1 (1-\tau)^{-\frac{1}{2}\big(\frac{n}{p} + \delta -\frac{n}{q}\big)}\tau^{\delta/2-1}\,ds \\
& \leq C(n, p,q, \delta,\delta',\delta^*)  t^{-\big(\frac{n}{2p}-\frac{n}{2q}\big)}e^{-t\tilde{\beta}} \int_0^1 (1-\tau)^{\big(1-\frac{1}{2}\big(\frac{n}{p} -\frac{n}{q}+ \delta\big)\big)-1}\tau^{\delta/2-1}\,ds \\
&= C(n, p,q, \delta,\delta',\delta^*)  B\bigg(\frac{\delta}{2}, 1-\frac{1}{2}\bigg(\frac{n}{p}  -\frac{n}{q} + \delta\bigg)\bigg)  t^{-\big(\frac{n}{2p}-\frac{n}{2q}\big)}e^{-t\tilde{\beta}},
\end{split}
\end{equation}
where we have used the substitution $\tau = s/t$ to compute the Beta integral. As well, by our assumption \eqref{deltassump},
\[
1-\frac{1}{2}\bigg(\frac{n}{p}  -\frac{n}{q}+ \delta\bigg) > 0,
\]
so both arguments appearing in the Beta function in \eqref{Gtd3} are greater than zero, as required for convergence. Thus $Gu$ has the desired decay rate and this fact combined with \eqref{u0td3} shows \eqref{td3}.

Next we prove \eqref{td4}, where in this case, we assume
\begin{equation} \label{td4asmp}
\frac{n}{2p}-\frac{n}{2q} \geq 1.
\end{equation}
Choose $0 <\varepsilon < \text{min}\big\{1, \frac{1}{2}\big(1+\frac{n}{q}\big)\big\}$ and define
\[
p' = \frac{qn}{2(1-\varepsilon)q+n}.
\]
Then
\[
\frac{n}{2p'} - \frac{n}{2q} = 1-\varepsilon< 1,
\]
and this quantity can be taken arbitrarily close to 1 by choosing $\varepsilon$ small enough. We next claim $p < p' < n$. For the first inequality, note that by \eqref{td4asmp},
\[
0< 1- \varepsilon < 1 \leq  \frac{n}{2p}-\frac{n}{2q}.
\]
Rearranging this, we get
\begin{equation}
\begin{split}
\frac{1}{p} &> \frac{2(1-\varepsilon)}{n}+\frac{1}{q}\\
&= \frac{2(1-\varepsilon)q+n}{qn} \\
&= \frac{1}{p'},
\end{split}
\end{equation}
which implies $p' >p$. To show the second inequality $p'< n$, we use that
\[
\varepsilon < \frac{1}{2}\bigg(1+\frac{n}{q}\bigg),
\]
which after rearrangement implies
\begin{equation}
\begin{split}
\frac{1}{n} &<\frac{2(1-\varepsilon)}{n}+\frac{1}{q}\\
&= \frac{2(1-\varepsilon)q+n}{qn} \\
&= \frac{1}{p'},
\end{split}
\end{equation}
so that $p' < n$ follows.

Having shown that $p < p' < n$, it follows by interpolation that, if the initial condition $a \in L^p (\mathbb{H}^n) \cap L^n(\mathbb{H}^n)$, then $a \in L^{p'} (\mathbb{H}^n)$. So as long as 
\begin{equation} \label{p'ltq}
p' \leq q,
\end{equation}
then the steps used to derive \eqref{td3} are valid for $p'$ in place of $p$, which verifies \eqref{td4}. Thus it remains to show \eqref{p'ltq}, but since $0 < \varepsilon < 1$ by definition, then $0 < 2(1-\varepsilon)/n$, giving
\begin{equation}
\begin{split}
\frac{1}{q} &< \frac{2(1-\varepsilon)}{n} + \frac{1}{q} \\
&=\frac{2(1-\varepsilon)q+n}{qn} \\
& = \frac{1}{p'},
\end{split}
\end{equation}
so that $p' < q$.

The decay estimates for $\nabla u$ are proven in an analogous way. 

 
 \section{Proof of Theorem \ref{Lptimedecay}}  \label{Lptimedecaypf} To prove Theorem \ref{Lptimedecay}, we must show \eqref{td2}, and for this, we use result \eqref{lpspace} from Theorem \ref{Lpthm} to get
\begin{equation}
\|u(t)\|_{L^p(\mathbb{H}^n)} \leq C\big(n, p, \delta, \|a\|_{L^n(\mathbb{H}^n)} \big) e^{-t \beta^*},
\end{equation}
where for $0 < \delta < 1$ is chosen so that $n/p+\delta <n$, $C\big(n, p, \delta, \|a\|_{L^n(\mathbb{H}^n)} \big)$ is the constant appearing in \eqref{nnn}, and $\beta^* > 0$ is defined by \eqref{betastar}. Therefore,
\begin{equation}
\begin{split}
\lim_{t \to \infty} \|u(t)\|_{L^p(\mathbb{H}^n)} &\leq \lim_{t \to \infty} C\big(n, p, \delta, \|a\|_{L^n(\mathbb{H}^n)} \big) e^{-t \beta^*} \\
& = 0.
\end{split}
\end{equation}

\appendix
\section{Strong continuity of the semigroup $e^{tL}$} \label{appendix} In this section, we prove $e^{tL}$ is a contractive and strongly continuous semigroup on $L^p(T^*\mathbb{H}^n)$, $1 \leq p < \infty$. To do so, we first state two definitions, a theorem, and a corollary from \cite{reedsimon2}.

\begin{definition}\cite{reedsimon2} 
Let $X$ be a Banach space, $\varphi \in X$. An element $\ell \in X^*$ that satisfies $\|\ell\|_{X^*} = \|\varphi\|_{X}$ and $\ell(\varphi) = \|\varphi\|_{X}^2$ is called a normalized tangent functional to $\varphi$. 
\end{definition}

\begin{remark}
As mentioned in \cite{reedsimon2}, every $\varphi \in X$ has at least one normalized tangent functional by the Hahn-Banach Theorem.
\end{remark}

\begin{remark}
For a $\sigma$-finite measure space $(X, \mathscr{B}, m)$, it is straightforward to check that for any $u \in L^p(X)$, $1 \leq p < \infty$, there exists a normalized tangent functional given by
\begin{equation} \label{lpnormtan}
\ell = c|u|^{p-2}u,
\end{equation}
where $c = \|u\|_{L^p(X)}^{1-p/p'}$ and $p'$ is the H{\"o}lder conjugate to $p$.
\end{remark}

\begin{definition} \cite{reedsimon2} 
A densely defined operator $A$ on a Banach space $X$ is called accretive if for each $\varphi \in D(A)$, 
\begin{equation}
\text{Re}(\ell(A\varphi)) \geq 0
\end{equation}
for some normalized tangent functional $\ell$ to $\varphi$.
\end{definition}

With these definitions in hand, we can now state a Theorem from \cite{reedsimon2}, as well as a useful corollary, which will be the main tools in showing contractivity and strong continuity of $e^{tL}$  on $L^p(\mathbb{H}^n)$, $1 \leq p < \infty$.

\begin{thm}  \cite{reedsimon2} \label{scchar}
A closed operater $A$ on a Banach space $X$ is the generator of a strongly continuous and contractive semigroup $e^{-tA}$ if and only if $A$ is accretive and
\begin{equation}
\text{Ran} (\lambda_0 +A) = X
\end{equation}
for some $\lambda_0 >0$.
\end{thm}

\begin{cor} \cite{reedsimon2}   \label{scchar1}
Let $A$ be a closed operator on a Banach space $X$ such that both $A$ and its adjoint $A^*$ are accretive operators. Then $A$ generates a strongly continuous and contractive semigroup $e^{-tA}$.
\end{cor}

We now state the main result to be proven.

\begin{thm} \label{etLsc}
For $\mathbb{H}^n$, $n\geq 2$, the unique self-adjoint extension of $L$, which will also be denoted by $L$, generates a contractive and strongly continuous semigroup $e^{tL}$ on $L^p(T^*\mathbb{H}^n)$ for all $ 1 \leq p < \infty$.
\end{thm}

\begin{proof}
Since the extension $L$ is self-adjoint by Strichartz' work in \cite{Strichartz_laplacian}, it suffices by Corollary \ref{scchar1} to show $-L$ is accretive. To that end, let $u \in L^p(T^*M)$, $1 \leq p < \infty$. Then 
\begin{equation} \label{Laccr}
\begin{split}
\langle |u|^{p-2}u, Lu \rangle &= \langle |u|^{p-2}u,\Delta_{B,1}u \rangle - (n-1) \langle |u|^{p-2}u, u \rangle,
\end{split}
\end{equation}
where $\langle \cdot, \cdot \rangle$ represents the duality pairing. The term $(n-1) \langle |u|^{p-2}u, u \rangle \geq 0$, so if $-\Delta_{B,1}u$ is accretive for $1 \leq p < \infty$, then $\langle |u|^{p-2}u,\Delta_{B,1}u \rangle \leq 0$ and it follows that 
\begin{equation} 
\langle |u|^{p-2}u, Lu \rangle \leq 0.
\end{equation}
Therefore, if $-\Delta_{B,1}$ is accretive on $L^p(T^*M)$ for $1 \leq p < \infty$, then $-L$ is also accretive on $L^p(T^*M)$ for all $ 1 \leq p < \infty$ and by Corollary \ref{scchar1}, $e^{tL}$ extends to a strongly continuous semigroup on $L^p(T^*M)$ for all $1 \leq p < \infty$.

So to conclude the proof, we must show that $-\Delta_{B,1}u$ is accretive for $1 \leq p < \infty$. To do this, we will use Theorem \ref{scchar} in reverse and show that $-\Delta_{B,1}u$ generates a contractive and strongly continuous semigroup.  

First we state as a lemma a pointwise domination result relating the semigroup $e^{t\Delta_{B,k}}$ on $(k,0)$-tensors to the semigroup $e^{t\Delta_{g}}$ for functions. 

\begin{lem} \cite{guneysu} \cite{hess1980} \cite{rosethesis} \label{katosimon}
Let $(M,g)$ be a complete Riemannian manifold. For all $k = 1, 2, 3, \ldots, n$ and all $\omega \in L^2 (\Omega^k(M))$,
\begin{equation}
|e^{t\Delta_{B,k}}\omega|(x) \leq e^{t\Delta_g}|\omega|(x), \quad \text{for every $x \in M$}.
\end{equation}
\end{lem}

\begin{remark}This lemma was shown for compact manifolds by Hess, Schrader, and Uhlenbrock in \cite{hess1980} , where the authors also remark on extensions to the noncompact case under certain assumptions. G{\"u}neysu then proved the general noncompact case in \cite{guneysu}. For more information on this inequality, see \cite{rosethesis}.
\end{remark}

Next we state a definition from Shigekawa in \cite{shigekawa}.

\begin{definition}
Let $(X, \mathscr{B},m)$ by a $\sigma$-finite measure space and suppose $ T_t$ is a strongly continuous and contractive semigroup on $L^p(X, \R)$, $1 \leq p < \infty$. Then we say $ T_t$ is a Markovian semigroup if, for every $f \in L^p(X, \R)$ with $0 \leq f \leq 1$, 
\begin{equation}
0 \leq T_t f \leq 1.
\end{equation}
\end{definition}

As well, we will need the following theorem that is related to the pointwise domination estimate in Lemma \ref{katosimon}. This result also comes from Shigekawa in \cite{shigekawa} and discussed in further detail by Ouhabaz in \cite{ouhabaz}.

\begin{thm} \cite{ouhabaz}\cite{shigekawa}  \label{shige}
Suppose that $\vec{T}_t$ is a vector valued semigroup which is strongly continuous and contractive on $L^2(X, K)$, where $K$ is a real or complex separable Hilbert space. If for every $u \in L^2(X,K)$,
\begin{equation}
|\vec{T}_t u|(x) \leq T_t|u|(x) \quad x \in X,
\end{equation}
where $T_t$ is a strongly continuous, contractive, and Markovian semigroup on $L^p(X, \R)$ for all $1 \leq p < \infty$, then $\vec{T}_t$ extends to a strongly continuous and contractive semigroup on $L^p(X, K)$ for all $1 \leq p < \infty$.
\end{thm}

In light of this theorem and Lemma \ref{katosimon}, it remains to show two things. The first is that $e^{t\Delta_{B,1}}$ is strongly continuous and contractive on $L^2(T^*\mathbb{H}^n)$. The second is that $e^{t\Delta_g}$ is a strongly continuous, contractive, and Markovian semigroup on $L^p(\mathbb{H}^n)$ for all $1 \leq p < \infty$.

For the first statement, this follows by Strichartz' Theorem 3.7 in \cite{Strichartz_laplacian}, which gives that $e^{t\Delta_{B,1}}$  extends to a strongly continuous and contractive semigroup on $L^p(T^*\mathbb{H}^n)$ for $\dfrac{3}{2} \leq p \leq 3$ and, in particular, on $L^2(T^*\mathbb{H}^n)$.

For the second statement, this also follows by the work of Strichartz in  \cite{Strichartz_laplacian}, specifically Theorems 3.5 and 3.6. These theorems asserts that $e^{t\Delta_g}$ extends to a strongly continuous and contractive semigroup on $L^p(\mathbb{H}^n)$ for all $1 \leq p < \infty$ and that for all $f \in L^2(\mathbb{H}^n)$,
\begin{equation}
e^{t\Delta_g}f(x) = \int_{\mathbb{H}^n} H_t(x,y)f(y)\,dV(y),
\end{equation}
where $H_t(x,y)$ is a $C^{\infty}$ real-valued function on $[0, \infty) \times \mathbb{H}^n \times \mathbb{H}^n$ such that $H_t(x,y) = H_t(y,x)$, $H_t(x,y) >0$ for all $x, y \in \mathbb{H}^n$ and $t>0$, and that
\begin{equation}
\int_{\mathbb{H}^n} H_t(x,y)\,dV(y) \leq 1
\end{equation} 
for all $x \in \mathbb{H}^n$ and $t >0$.

With these facts stated, it follows that if $f \in L^p(\mathbb{H}^n)$ such that $0 \leq f \leq 1$, where $1 \leq p < \infty$, then
\begin{equation}
e^{t\Delta_g}f(x) = \int_{\mathbb{H}^n} H_t(x,y)f(y)\,dV(y) \geq 0
\end{equation}
and
\begin{equation}
e^{t\Delta_g}f(x) = \int_{\mathbb{H}^n} H_t(x,y)f(y)\,dV(y) \leq \int_{\mathbb{H}^n} H_t(x,y)\,dV(y) \leq 1,
\end{equation}
so that $\{e^{t\Delta_g}\}_{t\geq0}$ is also Markovian.

Thus by Lemma \ref{katosimon} and Theorem \ref{shige}, we conclude $e^{t\Delta_{B,1}}$ extends to a strongly continuous and contractive semigroup on $L^p(T^*\mathbb{H}^n)$ for $1 \leq p < \infty$, 
\end{proof}

\bibliography{ref}

\def\cprime{$'$} \def\cprime{$'$} \newcommand{\SortNoop}[1]{}
\begin{thebibliography}{10}

\bibitem{math_meth_phys}
George~B. Arfken, Hans~J. Weber, and Frank~E. Harris.
\newblock {\em Mathematical methods for physicists: a comprehensive guide}.
\newblock Academic Press, seventh edition, 1987.

\bibitem{auscher2004stability}
P~Auscher, S~Dubois, and P~Tchamitchian.
\newblock On the stability of global solutions to {N}avier--{S}tokes equations
  in the space.
\newblock {\em Journal de math{\'e}matiques pures et appliqu{\'e}es},
  83(6):673--697, 2004.

\bibitem{buckmaster}
Tristan Buckmaster and Vlad Vicol.
\newblock Nonuniqueness of weak solutions to the {N}avier-{S}tokes equation.
\newblock {\em Annals of Mathematics}, 189(1):101--144, 2019.

\bibitem{Chan-C-Disconzi}
Chi~Hin Chan, Magdalena Czubak, and Marcelo Disconzi.
\newblock The formulation of the {N}avier-{S}tokes equations on {R}iemannian
  manifolds.
\newblock {\em J. Geom. Phys.}, 121:335--346, 2017.

\bibitem{conway}
John~B. Conway.
\newblock {\em A course in functional analysis}.
\newblock Springer, second edition, 1990.

\bibitem{geogroup}
Cornelia Dru{\c t}u and Michael Kapovich.
\newblock {\em Geometric group theory}, volume~63 of {\em Colloquium
  Publications}.
\newblock American Mathematical Society, 2018.

\bibitem{ebin_mars}
David~G. Ebin and Jerrold Marsden.
\newblock Groups of diffeomorphisms and the motion of an incompressible fluid.
\newblock {\em Ann. of Math. (2)}, 92(1):102--163, 1970.

\bibitem{fabes-jones-riviere}
E.~B. Fabes, B.~F. Jones, and N.~M. Riviere.
\newblock The initial value problem for the {N}avier-{S}tokes equations with
  data in {$L^p$}.
\newblock {\em Archive for Rational Mechanics and Analysis}, 45(3):222--240,
  1972.

\bibitem{flr}
E.~B. Fabes, J.~E. Lewis, and N.~M. Riviere.
\newblock Boundary value problems for the navier-stokes equations.
\newblock {\em American Journal of Mathematics}, 99(3):626--668, 1977.

\bibitem{kato_fujita}
Hiroshi Fujita and Tosio Kato.
\newblock On the non stationary {N}avier-{S}tokes system.
\newblock {\em Arch. Rational Mech. Anal.}, 16:269--315, 1961.

\bibitem{Furioli2000}
Giulia Furioli, Lemari{\'e}-Rieusset, Pierre Gilles, and Elide Terraneo.
\newblock Unicit{\'e} dans {$L^3(R^3)$} et d'autres espaces fonctionnels
  limites pour {N}avier-{S}tokes.
\newblock {\em Revista Matem{\'a}tica Iberoamericana}, 16(3):605--667, 2000.

\bibitem{gall}
Isabelle Gallagher, Dragos Iftimie, and Fabrice Planchon.
\newblock Asymptotics and stability for global solutions to the
  {N}avier-{S}tokes equations.
\newblock {\em Annales de l'Institut Fourier}, 53(5):1387--1424, 2003.

\bibitem{giga_semilinear}
Yoshikazu Giga.
\newblock Solutions for semi linear parabolic equations in {$L^p$} and
  regularity of weak solutions of the {N}avier-{S}tokes system.
\newblock {\em Journal of Differential Equations}, 61:186--212, 1986.

\bibitem{guneysu}
Batu G{\"u}neysu.
\newblock {\em Covariant Schr{\"o}dinger semigroups on Riemannian manifolds}.
\newblock Birkh{\"a}user Basel, first edition, 2017.

\bibitem{hess1980}
H.~Hess, R.~Schrader, and D.~A. Uhlenbrock.
\newblock Kato's inequality and the spectral distribution of {L}aplacians on
  compact {R}iemannian manifolds.
\newblock {\em J. Differential Geom.}, 15(1):27--37, 1980.

\bibitem{Hopf}
Eberhard Hopf.
\newblock \"{U}ber die anfangswertaufgabe f\"ur die hydrodynamischen
  grundgleichungen.
\newblock {\em Math. Nachr.}, 4:213--231, 1951.

\bibitem{iwashita}
Hirokazu Iwashita.
\newblock Lq-lrestimates for solutions of the nonstationary stokes equations in
  an exterior domain and the navier-stokes initial value problems inlqspaces.
\newblock {\em Mathematische Annalen}, 285(2):265--288, 1989.

\bibitem{jost}
J{\"u}rgen Jost.
\newblock {\em Riemannian geometry and geometric analysis}.
\newblock Springer-Verlag Berlin Heidelberg, first edition, 1995.

\bibitem{KatoMild}
Tosio Kato.
\newblock Strong {$L^{p}$}-solutions of the {N}avier-{S}tokes equation in
  {${R}^{m}$}, with applications to weak solutions.
\newblock {\em Math. Z.}, 187(4):471--480, 1984.

\bibitem{kato1995}
Tosio Kato.
\newblock {\em Perturbation Theory for Linear Operators}.
\newblock Classics in Mathematics. Springer Berlin Heidelberg, 1995.

\bibitem{kuka}
Igor Kukavica.
\newblock Pressure integrability conditions for uniqueness of mild solutions of
  the {N}avier-{S}tokes system.
\newblock {\em J. Differential Equations}, 223(2):427--441, 2006.

\bibitem{lem-rie}
P.G. Lemari{\'e}-Rieusset.
\newblock {\em Recent developments in the Navier-Stokes problem}, volume 431 of
  {\em Chapman \& Hall/CRC Research Notes in Mathematics}.
\newblock Chapman \& Hall/CRC, Boca Raton, FL., 2002.

\bibitem{lem-rie2}
P.G. Lemari{\'e}-Rieusset.
\newblock {\em The Navier-Stokes Problem in the 21st Century}.
\newblock Chapman \& Hall/CRC, Boca Raton, FL., 04 2016.

\bibitem{LerayExt}
Jean Leray.
\newblock Etudes de diverses equations integrales non lineaires et de quelques
  problemes que pose l'hydrodynamique.
\newblock {\em J. Math. Pures Appl.}, 12:1--82, 1933.

\bibitem{leray1934}
Jean Leray.
\newblock Sur le mouvement d'un liquide visqueux emplissant l'espace.
\newblock {\em Acta Math.}, 63:193--248, 1934.

\bibitem{Lew}
Jeff Lewis.
\newblock The initial-boundary value problem for the {N}avier-{S}tokes
  equations with data in {$L^p$}.
\newblock {\em Indiana Univ. Math. J.}, 22:739--761, 1973.

\bibitem{lions}
P.L. Lions and N.~Masmoudi.
\newblock Uniqueness of mild solutions of the {N}avier-{S}tokes system in
  {$L^N$}.
\newblock {\em Communications in Partial Differential Equations},
  26(11-12):2211--2226, 2001.

\bibitem{lohoue}
N{\"o}el Lohou{\'e}.
\newblock Estimation des projecteurs de {D}e {R}ham {H}odge de certaines
  vari{\'e}t{\'e}s {R}iemaniennes non compactes.
\newblock {\em Math. Nchr.}, 279(3):272--298, 2006.

\bibitem{masuda1984weak}
Kyuya Masuda.
\newblock Weak solutions of {N}avier-{S}tokes equations.
\newblock {\em Tohoku Mathematical Journal, Second Series}, 36(4):623--646,
  1984.

\bibitem{MitreaTaylor}
Marius Mitrea and Michael Taylor.
\newblock Navier-{S}tokes equations on {L}ipschitz domains in {R}iemannian
  manifolds.
\newblock {\em Math. Ann.}, 321(4):955--987, 2001.

\bibitem{miyakawa1981}
Tetsuro Miyakawa.
\newblock On the initial value problem for the navier-stokes equations in
  {$L^p$} spaces.
\newblock {\em Hiroshima Math. J.}, 11(1):9--20, 1981.

\bibitem{ouhabaz}
El~Maati Ouhabaz.
\newblock {$L^p$} contraction semigroups for vector valued functions.
\newblock {\em Positivity}, 3(1):83--93, 1999.

\bibitem{Pierfelice}
Vittoria Pierfelice.
\newblock The incompressible {N}avier-{S}tokes equations on non-compact
  manifolds.
\newblock {\em J. Geom. Anal.}, 27, 06 2014.

\bibitem{reedsimon1}
Michael Reed and Barry Simon.
\newblock {\em Methods of modern mathematical physics I: functional analysis}.
\newblock Academic Press, first edition, 1975.

\bibitem{reedsimon2}
Michael Reed and Barry Simon.
\newblock {\em Methods of modern mathematical physics II: {F}ourier analysis,
  self-adjointness}.
\newblock Academic Press, first edition, 1975.

\bibitem{rosethesis}
Christian Rose.
\newblock Heat kernel estimates based on {R}icci curvature integral bounds.
\newblock {\em Ph.D. Thesis, Technische Universit{\"a}t Chemnitz}, 2017.

\bibitem{rudinfn}
Walter Rudin.
\newblock {\em Functional analysis}.
\newblock McGraw-Hill, second edition, 1991.

\bibitem{schonbek19852}
Maria~Elena Schonbek.
\newblock ${L}^2$ decay for weak solutions of the {N}avier-{S}tokes equations.
\newblock {\em Archive for rational mechanics and analysis}, 88(3):209--222,
  1985.

\bibitem{setti}
Alberto~G. Setti.
\newblock A lower bound for the spectrum of the {L}aplacian in terms of
  sectional and {R}icci curvature.
\newblock {\em Proceedings of the American Mathematical Society},
  112(1):277--282, 1991.

\bibitem{shigekawa}
Ichiro Shigekawa.
\newblock {$L^p$} contraction semigroups for vector valued functions.
\newblock {\em J. Funct. Anal.}, 147(1):69--108, 1997.

\bibitem{sobolevskii}
P.~E. Sobolevski\u{\i}.
\newblock Non-stationary equations of viscous fluid dynamics.
\newblock {\em DOKLADY AKADEMII NAUK SSSR}, 128(1):45--48, 1959.

\bibitem{Stein}
Elias~M. Stein.
\newblock {\em Singular integrals and differentiability properties of
  functions}.
\newblock Princeton Mathematical Series, No. 30. Princeton University Press,
  Princeton, N.J., 1970.

\bibitem{Strichartz_laplacian}
Robert~S. Strichartz.
\newblock Analysis of the {L}aplacian on the complete {R}iemannian manifold.
\newblock {\em J. Funct. Anal.}, 52(1):48--79, 1983.

\bibitem{michael1999partial}
Michael Taylor.
\newblock Partial differential equations. {III}.
\newblock {\em Applied Mathematical Sciences}, 117, 1999.

\bibitem{tsai-lec}
Tai-Peng Tsai.
\newblock {\em Lectures on Navier-Stokes equations}, volume 192 of {\em
  Graduate Studies in Mathematics}.
\newblock American Mathematical Society, Providence, RI, 2018.

\bibitem{von1980regularity}
Wolf Von~Wahl.
\newblock Regularity questions for the {N}avier-{S}tokes equations.
\newblock In {\em Approximation Methods for {N}avier-{S}tokes Problems}, pages
  538--542. Springer, 1980.

\bibitem{weissler}
Fred~B. Weissler.
\newblock The {N}avier-{S}tokes initial value problem in {$L^p$}.
\newblock {\em Archive for Rational Mechanics and Analysis}, 74(3):219--230,
  1980.

\end{thebibliography}
\bibliographystyle{plain}
\vspace{.125in}

\end{document}